\documentclass[pdflatex,sn-mathphys]{sn-jnl}
\usepackage[all,cmtip]{xy}

\jyear{2021}%

\theoremstyle{plain} 
\newtheorem{thm}{Theorem}[section] 
\newtheorem{cor}[thm]{Corollary} 
\newtheorem{lem}[thm]{Lemma} 
\newtheorem{prop}[thm]{Proposition}

\theoremstyle{definition} 
\newtheorem{defn}{Definition}[section] 
\newtheorem{exem}{Example} [section]
\newtheorem{notation}{Notation}[section]

\theoremstyle{remark} 
\newtheorem{rem}{Remark}

\begin{document}

\title[Uniform homotopy invariance of Roe Index of the signature operator]{Uniform homotopy invariance of Roe Index of the signature operator}


\author*[1,2]{\pfx{Dr} \fnm{Stefano} \sur{Spessato} \sfx{ORCID: 0000-0003-0069-5853}\email{stefano.spessato@unicusano.it }}

\affil*[1]{\orgname{Università degli Studi "Niccolò Cusano"}, \orgaddress{\street{via don Carlo Gnocchi 3}, \city{Rome}, \postcode{00166}, \country{Italy}}}


\abstract{In this paper we study the Roe index of the signature operator of manifolds of bounded geometry. Our main result is the proof of the uniform homotopy invariance of this index. In other words we show that, given an orientation-preserving uniform homotopy equivalence $f:(M,g) \longrightarrow (N,h)$ between two oriented manifolds of bounded geometry, we have that  $f_\star(Ind_{Roe}D_M) = Ind_{Roe}(D_N)$. Moreover we also show that the same result holds considering a group $\Gamma$ acting on $M$ and $N$ by isometries and assuming that $f$ is $\Gamma$-equivariant. The only assumption on the action of $\Gamma$ is that the quotients are again manifolds of bounded geometry.}




\maketitle
\subsection*{Acknowledgments}
I would like to thank Paolo Piazza and Vito Felice Zenobi for their help and for the several discussion that we had. Moreover I also thank Thomas Schick for his suggestions and support.
\\
\\Data sharing not applicable to this article as no datasets were generated or analysed during the current study.

\section*{Introduction}
This is the second of two papers about uniform homotopy invariants. In the first one \cite{Spes} we studied the $L^{q,p}$-cohomology, in the second one we study the Roe Index of the signature operator. Two versions, with some stronger assumptions, of these works can be found in the Ph.D. thesis of the author \cite{thesis}.
\\
\\Let us introduce our geometric setting: we study oriented manifolds of bounded geometry, i.e. Riemannian manifolds with a lower bound on the injectivity radius and some upper bounds on the norms of the covariant derivatives of the sectional curvature. On these manifolds we consider a \textit{uniformly proper discontinuous and free} (or \textit{FUPD}) action of group $\Gamma$ of isometries. A FUPD action is an action such that the quotient space is a manifold of bounded geometry. 
\\We are interested in studying these manifolds up to uniform homotopy. A \textit{uniform homotopy equivalence} $f:(M,g) \longrightarrow (N,h)$ is a uniformly continuous homotopy equivalence admitting a homotopy inverse $g$ such that $g$ itself and the homotopies between the compositions and the identities are uniformly continuous. As proved in Proposition 1.7 of \cite{Spes}, we have that each uniform homotopy equivalence can be approximate by a smooth lipschitz map whose derivatives, in normal coordinates, are uniformly bounded. Moreover if $f$ is $\Gamma$-equivariant, then also its approximation is $\Gamma$-equivariant. All the definitions and properties needed about this geometric setting can be found in the first paper \cite{Spes}.
\\Let us denote by $\mathcal{L}^2(M)$ the space of squared-integrable complex forms over the oriented manifold of bounded geometry $(M,g)$. Let us suppose $n = dim(M)$. Then we can consider the \textit{signature operator} $D_M: dom(D_M) \subset \mathcal{L}^2(M) \longrightarrow \mathcal{L}^2(M)$. As showed in \cite{Anal}, the signature operator induces a class in the $[n]$-th\footnote{$[n] = 0$ if $n$ is even, $1$ otherwise} K-theory group of the Coarse algebra $C^*(M)^\Gamma$. This class is called the \textit{Roe Index} and it is denoted by $Ind_{Roe}D_M$. Our main goal in this paper is to prove that, given a $\Gamma$-equivariant uniformly homotopy equivalence $f:(M,g) \longrightarrow (N,h)$ between manifolds of bounded geometry, we have
\begin{equation}
f_\star (Ind_{Roe}D_M) = Ind_{Roe}D_N.
\end{equation}
This is the strategy of the proof.
\\Our first step is to introduce a specific coarse structure on the disjoint union of $(M,g)$ and $(N,h)$. Indeed, \textit{a priori}, $M \sqcup N$ doesn't have a metric coarse structure induced by the metric because it is not connected (and so it isn't a metric space). The existence of this coarse structure allow us to define the coarse algebra $C^*_f(M \sqcup N)^\Gamma$ and the structure algebra $D^*_f(M \sqcup N)^\Gamma$.
\\The second step is to prove some properties on smoothing operators between $\mathcal{L}^2$-spaces of manifolds of bounded geometry. In particular we prove some conditions on the kernel of an operator $A$ such that $A$ and $dA$ are bounded. Here $d$ is the minimal closure of the exterior derivative of compactly supported smooth forms.
\\The third step is to define a signature operator $D_{M \sqcup N}$ and to find a perturbation $P$ in $C^*_f(M \sqcup N)^\Gamma$ such that $D_{M \sqcup N} + P$ has a $\mathcal{L}^2$-bounded inverse. In order to find such a perturbation we follow the works of \cite{hils}, \cite{piazzaschick}, \cite{vito}, \cite{wahl}, \cite{fukum} and we have to define some operators in $C^*_f(M \sqcup N)^\Gamma$. One of this operators is the operator $T_f$ defined in \cite{Spes}.
\\In the last step we conclude by showing that
\begin{equation}
 Ind_{Roe}(D_{M \sqcup N}) = Ind_{Roe}(D_{M \sqcup N} + P) = 0  
\end{equation}
and by proving the existence of a morphism $H: K_{n}(C^*_f(M \sqcup N)^\Gamma) \longrightarrow K_{n}(C^*(N)^\Gamma)$  such that
\begin{equation}
H_\star(Ind_{Roe}(D_{M \sqcup N})) = f_\star(Ind_{Roe}(D_{M})) - Ind_{Roe}(D_{N}).
\end{equation}

\section{Coarse geometry}
	\subsection{Coarse structures}
The next definitions can be found in \cite{Anal}. Let $X$ be a set
	\begin{defn}
		A \textbf{coarse structure} over $X$ is a provision, for each set $S$ of an equivalence relation on the set of maps form $S$ to $X$. If $p_1,p_2: S \rightarrow X$ are in relation then they are said to be \textbf{close} and it is also required that if 
		\begin{itemize}
			\item if $p_1$ and $p_2: S \longrightarrow X$ are close and $q:S' \longrightarrow S$ is another map then $p_1 \circ q$ and $p_2 \circ q$ are close too.
			\item If $S= S_1 \cup S_2$, $p_1, p_2:S\longrightarrow X$ are maps whose restrictions to $S_1$ and $S_2$ are close, then $p_1$ and $p_2$ are close,
			\item two constant maps are always close to each other.
		\end{itemize}
		If $X$ has a coarse structure, then is a \textbf{coarse space}.
	\end{defn}
Before of introducing two examples of coarse structure, we need to recall some notions from the first paper \cite{Spes}. A map $f:(X,d_X) \longrightarrow (Y,d_Y)$ between metric spaces is \textit{uniform} if it is uniformly continuous and uniformly proper. With \textit{uniformly proper} we mean that the diameter of the preimage of a subset $A$ only depends on the diameter of $A$. Moreover, if there is a group $\Gamma$ acting on $(X,d_X)$ and on $(Y,d_Y)$, then two maps $f_1$ and $f_2:(X,d_X) \longrightarrow (Y,d_Y)$ are \textit{$\Gamma$-uniformly homotopic} and we denote it by $f_1 \sim_\Gamma f_2$ if there is a uniform homotopy $H$ between them.
	\begin{exem}\label{bru}
		If $(X,d_X)$ is a metric space, then it's possible to define a \textit{metric} coarse structure in the following way: let $p_1$ e $p_2: S \longrightarrow X$ be two functions, then
		\begin{equation}
		p_1 \sim p_2 \iff \exists C \in \mathbb{R} \mbox{           such that           }\forall x \in S \mbox{                         }d_X(p_1(x), p_2(x)) \leq C.
		\end{equation}
		Let us consider two maps $p_1$ and $p_2:(Y, d_Y) \longrightarrow (X, d_X)$ and let $\Gamma$ be a group which acting FUPD on $X$ and on $Y$. Let us suppose that $p_1$ and $p_2$ uniform maps such that $p_1 \sim_{\Gamma} p_2$, then we have that $ p_1 \sim p_2$. Indeed, since Proposition 1.6 and Lemma 1.9 of \cite{Spes}, there are two lipschitz maps $p_1'$ and $p_2'$ such that $p_i' \sim p_i$ and there is a lipschitz homotopy $H$ connecting $p_1'$ and $p_2'$. Then we have that $p_1' \sim p_2'$, indeed
		\begin{equation}
		d_X(p_1'(x), p_2'(x)) = d(h(x,0), h(x,1)) \leq C_H,
		\end{equation}
		where $C_H$  is the lipschitz constant of $H$. This implies that $p_1 \sim p_2$
	\end{exem}
\begin{exem}\label{coarseexem}
	Consider a Riemannian manifold $X := M \sqcup N$ where $M$ and $N$ are two connected Riemannian manifolds and let $f:M \longrightarrow N$ be a uniform-homotopy equivalence.
	\\In this case $X$ is not a metric space in a natural way and so it has not \textit{a priori} a metric coarse structure. We will define a coarse structure using the map $f$.
	\\Let $p_1$ e $p_2: S \longrightarrow X$ be two maps, then
	\begin{equation}
	p_1 \sim p_2 \iff \exists C \in \mathbb{R} \forall x \in S s.t. d_N((f \sqcup id_N)\circ p_1(x), (f \sqcup id_N)\circ p_2(x)) \leq C.
	\end{equation}
\end{exem}
	\begin{defn}
		Let $X$ be a coarse space and let $S \subseteq X\times X$. Then $S$ is \textbf{controlled} if the projections $\pi_1, \pi_2: S \subseteq X$ are close. A family of subsets $\mathcal{U}$ of $X$ is \textbf{uniformly bounded} if $\cup_{U \in \mathcal{U}}U \times U$ is controlled. A subset $B$ of $X$ is \textbf{bounded} if $B\times B$ is controlled.
	\end{defn}
	\begin{defn}
		Let $X$ be a locally compact topological space. A coarse structure on $X$ is \textbf{proper} if
		\begin{itemize}
			\item $X$ has an uniformly bounded open covering,
			\item every bounded subset of $X$ has compact closure.
		\end{itemize}
	\end{defn}
	\begin{defn}
		A coarse space $X$ is \textbf{separable} if $X$ admits a  countable, uniformly bounded, open covering.
	\end{defn}
	\begin{rem}
		Given a metric space $(X, d_X)$, then the metric coarse structure is proper and separable. Moreover if $M$ and $N$ are two connected Riemannian manifolds and $f:M \longrightarrow N$ is a uniform homotopy equivalence, then also the coarse structure defined in the Example \ref{coarseexem} is proper and separable.
	\end{rem}
	\subsection{Roe and Structure Algebras}
	Consider a locally compact topological space $X$ endowed with a proper and separable coarse structure. Moreover, let $\rho: C_0(X) \longrightarrow B(H)$ be a representation of $C_0(X)$, where $H$ is a separable Hilbert space. Then
	\begin{defn}
		The \textbf{support} of an element $v$ in $H$ is the complement in $X$ of all open sets $U$ such that $\rho(f)v = 0$ for all $f$ in $C_0(U)$.
	\end{defn}
	\begin{defn}
		The \textbf{support} of an operator $T$ in $B(H)$ is the complement in $X\times X$ of the union of $U\times V$ such that
		\begin{equation}
		\rho(f)T\rho(g) = 0,
		\end{equation}
		for all $f$ in $C_0(U)$ and for all $g$ in $C_0(V)$. An operator is \textbf{controlled} or \textbf{finite propagation} if its support is a controlled set.
	\end{defn}
	\begin{exem}
		Consider $(X,g)$ a Riemannian manifold. Let $H_X$ be Hilbert space $H_X := \mathcal{L}^2(X)$ and let $\rho$ be the representation of $C_0(X)$ on $H_X$ given by the point-wise multiplication. With respect to the metric coarse structure, an operator $T$ is controlled if and only if
		\begin{equation}
		\exists C \in \mathbb{R} \forall \phi, \psi \in C_0(X) s.t. d_X(supp(\phi), supp(\psi)) \geq C \implies \phi T \psi = 0.
		\end{equation}
	\end{exem}
\begin{exem}
	Consider $M$ and $N$ as in Example \ref{coarseexem} and let us impose $H_X := \mathcal{L}^2(M\sqcup N)$. Fix the point-wise multiplication  as representation of $C_0(X)$ on $H_X$. Then an operator $T$ is controlled if there is a number $C > 0$ such that for each $\phi, \psi \in C_0(X)$ we have that
	\begin{equation}
	d_N((f \sqcup id_N)(supp(\phi)), (f \sqcup id_N)(supp(\psi))) \geq  C \implies \phi T \psi = 0.
	\end{equation}
\end{exem}
The next two definitions can be found in \cite{siegel}.
\begin{defn}
	Let $X$ be a coarse space. Consider $H$ be a Hilbert space equipped with a representation $\rho: C_0(X) \longrightarrow B(H)$, a group $\Gamma$ and a unitary representation $U : \Gamma \longrightarrow B(H)$. We say that the the triple $(H, U, \rho)$ is a \textbf{$\Gamma$-equivariant $X$-module} or simply a \textbf{$(X, \Gamma)$-module} if 
	\begin{equation}
		U(\gamma)\circ \rho(f) = \rho(\gamma^*(f)) \circ U(\gamma),
	\end{equation}
	for every $\gamma \in \Gamma$, $f \in C_0(X)$.
\end{defn}
\begin{defn}
A represention $\rho$ of a group $\Gamma$ in $B(H)$ for some Hilbert space $H$ is \textbf{nondegenerate} if the set
\begin{equation}
\{\rho(\gamma)h \in H \vert \gamma \in \Gamma, h \in H\}
\end{equation}
is dense in $H$.
\end{defn}
\begin{defn}
	A representation $\rho$ of a group $\Gamma$ in $B(H)$ for some Hilbert space $H$ is \textbf{ample} if it is nondegenerate and $\rho(g)$ is a compact operator if and only if $g = 0$.
	\\If the Hilbert space $H$ is separable and the representation $\rho$ is the countable direct sum of a fixed ample representation, then $\rho$ is said to be \textbf{very ample}.
\end{defn} 
\begin{exem}\label{perna}
	Let us consider a Riemannian manifold $X$ with $dim(X) > 0$ and consider the Hilbert space $H := \mathcal{L}^2(X)$. Let us define the representation $\rho_X: C_0^\infty(X) \longrightarrow B(\mathcal{L}^2(X))$ for each $\phi$ in $C_0^\infty(X)$ as 
	\begin{equation}
	\rho_X(\phi)(\alpha) := \phi \cdot \alpha
	\end{equation}
	where $\alpha$ is a differential form $\alpha$ in $\mathcal{L}^2(X)$. Observe that, because of Hadamard-Schwartz inequality \cite{Sbordone}, $\rho_X(\phi)$ is a $\mathcal{L}^2$-bounded operator. In particular we have that $\rho$ is an ample representation of $C_0^\infty(X)$.
	\\Let us consider a subgroup of the isometries of $X$ called $\Gamma$. Then, the representation
	\begin{equation}
		\begin{split}
			U : \Gamma &\longrightarrow B(\mathcal{L}^2(X))\\
			\gamma &\longrightarrow \gamma^*.
		\end{split}
	\end{equation}
	is well defined since $\gamma^*$ is a $\mathcal{L}^2$-bounded operator.
	\\So the triple $(\mathcal{L}^2(X), U, \rho)$ is a $(X, \Gamma)$-module.
\end{exem}
\begin{exem}\label{serna}
Consider $(X,g)$ an oriented Riemannian manifold (possibly not connected). Consider the Hilbert space $H_X := \mathcal{L}^2(X) \otimes l^2(\mathbb{N})$. We have that an element in $H_X$ can be seen as sequence of $\{\alpha_i\}$ where $\alpha_i \in \mathcal{L}^2(X)$ such that
\begin{equation}
\sum_{n \in \mathbb{N}} \vert\vert\alpha_i\vert\vert^2 < + \infty.
\end{equation}
Consider, moreover, $\rho_X: C^{\infty}_{0}(X) \longrightarrow B(H_X)$ for each $\phi$ in $C^{\infty}_{0}(X)$ and for each $\{\alpha_i\}$ as
\begin{equation}
\rho_X(\phi)(\{\alpha_i\}) := \{\phi \cdot \alpha_i\}.
\end{equation}
Consider $\Gamma$ as a subgroup of isometries of $X$. Then we can define a representation $U_X: \Gamma \longrightarrow B(H_X)$ as
\begin{equation}
U_X(\gamma)(\{\alpha_i\}) := \{\gamma^* \alpha_i\}.
\end{equation}
Then $(H_X, \rho_X, U_X)$ defined in this way is a very ample $(X, \Gamma)$-module.
\end{exem}
\begin{exem}\label{terna}
Let $(X,g)$ be an oriented Riemannian manifold such that $dim(X)$ is even. Let us consider the Hodge star operator $\star$. Then we can define the chirality operator $\tau$ as $\tau \alpha := i^{\frac{n}{2}}\star$ if $n$ is even and $\tau \alpha := i^{\frac{n +1}{2}}\star$ otherwise. We have that $\tau$ defines, for each $p$ in $X$, an involution $\tau_p: \Lambda^*_p(X) \longrightarrow \Lambda^*_p(X)$ where $\Lambda^*_p(X)$ is the fiber in $p$ of the complexified of the exterior bundle of $X$. Let us denote by $V_{\pm}$ as the bundles whose fibers are the $\pm1$-eigenspaces of $\tau_p$.
\\Let us define the Hilbert space $\mathcal{L}^2(V_{\pm})$ as the closure of the space of compactly supported section $\Gamma_c(M, V_{\pm})$ respect to the bundle metric induced by the Riemannian metric $g$.
\\We obtain an orthogonal splitting $\mathcal{L}^2(M) = \mathcal{L}^2(V_{+}) \oplus \mathcal{L}^2(V_{-}).$
\\Let us define the vector space
\begin{equation}
H_X := \bigoplus_{i \in \mathbb{Z}_{<0}} \mathcal{L}^2(V_{-}) \oplus \bigoplus_{i \in \mathbb{N}} \mathcal{L}^2(V_{+}).
\end{equation}
An element in $H_X$ is a sequence $\{\alpha_i\}$ indexed by $i \in \mathbb{Z}$ of $\alpha_i \in \mathcal{L}^2(V_{-})$ if $i < 0$ and $\alpha_i \in \mathcal{L}^2(V_{+})$ if $i > 0$ such that
\begin{equation}
\sum_{i \in \mathbb{Z}} \vert\vert\alpha_i\vert\vert_{\mathcal{L}^2(X)}^{2} < + \infty.
\end{equation}
Consider, moreover the representation $\rho_X: C^\infty_0(X) \longrightarrow B(H_X)$ defined for each $\phi$ in $C^{\infty}_0(X)$ as
\begin{equation}
\rho_X(\phi)(\{\alpha_i\}) := \{\phi \cdot \alpha_i\}
\end{equation}
and, given a subgroup $\Gamma$ of isometries of $X$, let $U: \Gamma \longrightarrow B(H_X)$ be
\begin{equation}
U_X(\gamma)(\{\alpha_i\}) := \{\gamma^* \alpha_i\}.
\end{equation}
Then we have that $(H_X, \rho_X, U_X)$ is a very ample $(X, \Gamma)$-module.
\end{exem}
	\begin{defn}
		Let us consider a coarse space $X$ and let $(H, U, \rho)$ be a $\Gamma$-equivariant $X$-module. Then an operator $T$ in $B(H)$ is  \textbf{pseudo-local} if for all $f$ in $C_0(X)$ we have that $[\rho (f), T]$ is a compact operator.
	\end{defn}
	\begin{defn}
		An operator $T \in B(H)$ is \textbf{locally compact} if $T\rho(f)$ and $\rho(f)T$ are compact operators for all $f \in C_0(X)$.
	\end{defn}
	Fix a coarse space $X$ and a $\Gamma$-equivariant $X$-module $(H, U, \rho)$. We can define the following algebras.
	\begin{defn}
		The algebra $D^{\star}_{c, \rho}(X,H)$ is given by
		\begin{equation}
		\{T \in B(H)\vert T \textit{                    is pseudo-local and controlled               }\}.
		\end{equation}
	\end{defn}
	\begin{defn}
		We denote by $C^{\star}_{c, \rho}(X)$ the algebra
		\begin{equation}
		\{T \in D^{\star}_{c, \rho}(X.H)\vert T \textit{                    is locally compact               }\}.
		\end{equation}
	\end{defn}
	\begin{defn}
		Let $X$ be a coarse space, $(H, U, \rho)$ a $\Gamma$-equivariant $X$-module. Then we will denote by $D^{\star}_{c,\rho}(X,H)^{\Gamma}$ and $C^{\star}_{c, \rho}(X,H)^{\Gamma}$ the operators of $D^{\star}_{c,\rho}(X,H)$ and $C^{\star}_{c, \rho}(X,H)$ which commute with the action of $\Gamma$ on $H$.
	\end{defn}
	\begin{defn}
		The $C^*$-algebras $D^{\star}_\rho(X,H)$, $C^{\star}_\rho(X,H)$, $D^{\star}_\rho(X,H)^{\Gamma}$ and $C^{\star}_\rho(X,H)^{\Gamma}$ are the closure in $B(H)$, of $D^{\star}_{c,\rho}(X,H)$, $C^{\star}_{c,\rho}(X,H)$, $D^{\star}_{c,\rho}(X,H)^{\Gamma}$ e $C^{\star}_{c,\rho}(X,H)^{\Gamma}$. The algebra $D^{\star}_{\rho}(X,H)^{\Gamma}$ will be called \textbf{structure algebra} and $C^{\star}_{\rho}(X,H)^{\Gamma}$ will be called \textbf{coarse algebra} or \textbf{Roe algebra}.
	\end{defn}
	We have that $C^*_{\rho}(X,H)^{\Gamma}$ is an ideal of $D^*_{\rho}(X,H)^{\Gamma}$ and so the following sequence
	\begin{equation}
	{ 0 \longrightarrow C^{\star}_{\rho}(X,H)^\Gamma \longrightarrow D^{\star}_{\rho}(X,H)^{\Gamma} \longrightarrow \frac{D^{\star}_{\rho}(X,H)^{\Gamma}}{C^{\star}_{\rho}(X,H)^\Gamma} \longrightarrow 0 }
	\end{equation}
	is exact.
	\begin{rem}\label{boy}
	Let us consider a Riemannian manifold $(X,g)$ and let $\Gamma$ be a subgroup of isometries of $X$. In particular, if $X$ is connected, consider on $(X,g)$ the coarse metric structure. If $X = M \sqcup N$ and there is a uniform homotopy equivalence $f:M \longrightarrow N$ consider the coarse structure defined in Example \ref{coarseexem}. Let us consider the $(X, \Gamma)$-modules defined in Example \ref{perna} and \ref{serna}. Observe that $B(\mathcal{L}^2(X))$ can be embedded in $B(H_X)$ in the following way: for each $A$ in $B(\mathcal{L}^2(X))$, we define $\tilde{A} \in B(H_X)$ as follow
	\begin{equation}
	\tilde{A}(\{\alpha_i\}) := \{\beta_j\}
	\end{equation}
where $\beta_j = 0$ if $j\neq 0$ and $\beta_0 := A \alpha_0$.
\\We have that if $A$ is in $C^*_{\rho}(X, \mathcal{L}^2(X))^\Gamma$ then $\tilde{A}$ is in $C^*_{\rho_X}(X, H_X)^\Gamma$. Moreover if $A$ is in $D^*_{\rho}(X, \mathcal{L}^2(X))^\Gamma$ then $\tilde{A}$ is in $D^*_{\rho_X}(X, H_X)^\Gamma$. In the following sections, with a little abuse of notation, we will denote $\tilde{A}$ by $A$.
	\end{rem}
	\begin{rem}\label{girl}
	Let $(X,g)$ be a Riemannian manifold and $\Gamma$ be a subgroup of isometries of $X$. Let us suppose $dim(X)$ is even. Again if $X$ is connected we consider the metric coarse structure on $X$,  if $X = M \sqcup N$ and there is a uniform-homotopy equivalence $f:M \longrightarrow N$ we consider the coarse structure defined in Example \ref{coarseexem}. \\Fix on $X$ the $(X, \Gamma)$-modules defined in Example \ref{perna} and in Example \ref{terna}. 
	\\Let $A$ be an operator in $B(\mathcal{L}^2(X))$ such that $A(\mathcal{L}^2(V_+)) \subseteq \mathcal{L}^2(V_-)$. Let us define the operator $\tilde{A}$ as the operator
	\begin{equation}
	\tilde{A}(\{\alpha_i\}) := \{\beta_{j}\}
\end{equation}
where $\beta_j := \alpha_{j+1}$ if $j \neq -1$ and $\beta_{-1} := A(\alpha_0)$. Then if $A$ is in $D^*_{\rho}(X, \mathcal{L}^2(X))^\Gamma$ then $\tilde{A}$ is in $D^*_{\rho_X}(X, H_X)^\Gamma$. Moreover we also have that if $A$ and $B$ are operators in $D^*_{\rho}(X, \mathcal{L}^2(X))^\Gamma$ such that $A - B$ are in $C^*_{\rho}(X, \mathcal{L}^2(X))^\Gamma$, then we have that also $\tilde{A} - \tilde{B}$, which is the operator
\begin{equation}
[\tilde{A} - \tilde{B}](\{\alpha_i\}) := \{\beta_j\}
\end{equation}
where $\beta_j = 0$ if $j \neq -1$ and $\beta_{-1} = (A-B)\alpha_0$, is in $C^*_{\rho_X}(X, H_X)^\Gamma$. 
\\Moreover we also have that
\begin{equation}
\vert\vert\tilde{A} - \tilde{B}\vert\vert = \vert\vert A - B\vert\vert.
\end{equation}
In the following sections, with a little abuse of notation, we will denote $\tilde{A}$ with $A$.
\end{rem}	
	\begin{notation}
	Given a Riemannian manifold $(X,g)$, if we write $C^*(X)^\Gamma$ or $D^*(X)^\Gamma$ without specify $H$ and $\rho$, we are considering $H$, $\rho$ and $\Gamma$ as in Example \ref{perna}.
	\\
	\\Moreover, if $X = M \sqcup N$, as in Example \ref{coarseexem}, then the coarse structure depends on $f$. Thus, for this reason, we will denote its algebras by $C^*_f(M \sqcup N)^\Gamma$ and by $D_f(M \sqcup N)^\Gamma$.
	\end{notation}
	\subsection{Coarse maps}
	The following definitions and properties can be found in \cite{Anal} in Chapter 6.  
	\begin{defn}
		Let $X_1$ and $X_2$ be coarse spaces. A function $q:X_1 \longrightarrow X_2$ is called a \textbf{coarse map} if
		\begin{itemize}
			\item whenever $p$ and $p'$ are close maps into $X_1$, then so are the composition $q \circ p$ and $q \circ p'$,
			\item for every bounded (in the coarse sense of Definition 4.2) subset $B \subseteq X_2$ we have that $q^{-1}(B)$ is a bounded subset.
		\end{itemize}
	\end{defn}
	\begin{rem}
		If $(X_i, g_i)$ are manifolds of bounded geometry with the metric coarse structure, then a uniform map $f: (X_1, g_1) \longrightarrow (X_2, g_2)$ is a coarse maps.
		In order to prove this fact it is sufficient to prove that for each $R > 0$ and $p \in M$ there exist a number $S > 0$ and a $q \in N$ such that
		\begin{equation}\label{palle}
		f(B_{R}(p)) \subseteq B_{S}(q).
		\end{equation}
		Consider a $R$-ball on $X_1$ and fix $\epsilon$ and $\delta$ two $X_i$-small numbers such that
		\begin{equation}
		 d(x,y) < \delta \implies d(f(x),f(y)) \leq \epsilon.
		\end{equation}
		Observe that $B_{R}(p)$ can be covered by a finite number $K$ of balls of radius $\delta$. We can prove this fact adapting the proof of Proposition 2.16 of \cite{bound} and observing that, since the Bishop-Gromov inequality, there is a global bound on the Volume of a $(R + \delta)$-ball on a manifold of bounded geometry.
		\\Then we have that $f(B_{R}(p))$ is contained in $\bigcup\limits_{i= 0...K}B_{\epsilon}(q_i)$ and since it is connected we obtain that
		\begin{equation}
		diam(f(B_{R}(p)) \leq K \cdot 2 \epsilon =: S
		\end{equation}
		and so this means that there is a $q$ in $X_2$ such that (\ref{palle}) is satisfied. Moreover also the inclusion maps  $j_i: X_i \longrightarrow X_1 \sqcup X_2$ are coarse maps.
	\end{rem}
	Our next step is to introduce a morphism between the structure algebras of coarse spaces related to a coarse map. 
	\begin{defn}
		Let $X$ and $Y$ be proper separable coarse spaces and suppose that $C_0(X)$ and $C_0(Y)$ are non-degeneratly represented on Hilbert spaces $H_X$ and $H_Y$. Consider a coarse map $q: X \longrightarrow Y$. A bounded operator $V: H_X \longrightarrow H_Y$ \textbf{coarsely covers $q$} if the maps $\pi_1$ and $q \circ \pi_2$ from $supp(V) \subseteq X \times Y$ to $Y$ are close.
	\end{defn}
\begin{rem}\label{check}
As showed in Remark 6.3.10 of \cite{Anal}, if $f_0$ and $f_1$ are close maps and $V$ coarsely covers $f_0$, then it also coarsely covers $f_1$.
\end{rem}
\begin{defn}
	Let $X$ and $Y$ be coarse spaces, let $\rho_X : C_0(X) \longrightarrow B(H_X )$ and $\rho_Y : C_0(Y ) \longrightarrow B(H_Y )$ be ample representations on separable Hilbert spaces, and let $\phi: U \longrightarrow Y$ be a continuous proper map defined on an open subset $U \subseteq X$. An isometry $V : H_X \longrightarrow H_Y$ \textbf{topologically covers $\phi$} if for every $f \in C_0(Y)$ there is a compact operator $K$ such that in $B(H_Y )$
	\begin{equation}
		\rho_Y(f) = V\rho_X(\phi^*(f))V^* + K.
	\end{equation}
	An isometry $V$ which analytically and topologically covers a map $\phi$ then it \textbf{uniformly covers $\phi$}.
\end{defn}
Let us recall the Proposition 2.13 of \cite{siegel}
\begin{prop}
	Let $H_X$ be an ample $(X, \Gamma)$-module and let $H_Y$ be a very ample $(Y, \Gamma)$-module. Then every equivariant uniform map $\phi: X \longrightarrow Y$ is uniformly covered by an equivariant isometry $V : H_X \longrightarrow H_Y$.
\end{prop}
As consequence of this fact we can prove the following Lemma.
\begin{lem}
	Let $X$ and $Y$ be proper separable coarse spaces and fix $\Gamma$ a group. Let us consider a $\Gamma$-equivariant continuous coarse map $\phi: X \longrightarrow Y$ and consider an equivariant isometry $V : H_X \longrightarrow H_Y$ which uniformly covers $\phi$. Then the map
	\begin{equation}
		\begin{split}
			Ad_V: D_\rho^*(X, H_X)^\Gamma &\longrightarrow D_{\rho}^*(Y, H_Y)^\Gamma \\
			T &\longrightarrow VTV^*
		\end{split}
	\end{equation}
	is well defined, maps $C^*_\rho(X, H_X)^\Gamma$ in $C^*_\rho(Y, H_Y)^\Gamma$. Moreover the induced morphisms between the K-theory groups don't depend on the choice of $V$.
\end{lem}
	\begin{proof}
		Let us start by proving that $Ad_V(T)$ is in $D^*_\rho(Y, H_Y)^\Gamma$. In Lemma 6.3.11 of \cite{Anal}, the authors prove that if $T$ is a controlled operator, then $Ad_V(T)$ is controlled. Moreover, since $T$ and $V$ are both $\Gamma$-equivariant, also $Ad_V(T)$ is $\Gamma$-equivariant. Then we just have to prove that, given $f$ in $C_0(Y)$, then
		\begin{equation}
		[\rho_Y(f), VTV^*]
		\end{equation}
		is a compact operator. Observe that since $V$ topologically covers $\phi$, then
		\begin{equation}
		\rho_Y(f) = V\rho_X(\phi^*(f))V^* + K,
		\end{equation}
		where $K$ is a compact operator. This means that
		\begin{equation}
		\begin{split}
		[\rho_Y(f), VTV^*] &= [V\rho_X(\phi^*(f))V^*, VTV^*] + [K, VTV^*]\\ 
		&= V[\phi^*(f), T]V^* + [K, VTV^*],
		\end{split}
		\end{equation}
		which is a compact operator.
		\\
		\\That $Ad_V$ maps $C^*_\rho(X, H_X)^\Gamma$ in $C^*_\rho(Y, H_Y)^\Gamma$ is proved in Lemma 6.3.11 of \cite{Anal}. The $\Gamma$-equivariance, again, follows by the $\Gamma$-equivariance of $V$.
		\\Since Proposition 6.3.12 of \cite{Anal}, we have that the morphisms induced between $K_\star(C^*_\rho(X, H_X)^\Gamma)$ and $K_\star(C^*_\rho(Y, H_Y)^\Gamma)$ don't depend on the choice of $V$. Applying the same arguments used in the proof of Lemma 5.4.2. of \cite{Anal}, we obtain that the same holds for the morphisms between $K_\star(D^*_\rho(X, H_X)^\Gamma)$ and $K_\star(D^*_\rho(Y, H_Y)^\Gamma)$. Let us consider, indeed, a projection (or a unitary) $T$ in $D^*_\rho(X, H_X)^\Gamma$ and consider two isometries $V_1$ and $V_2$ that uniformly cover $f$.
		Consider the matrices 
		\begin{equation}
			\mathcal{V}_1 := \begin{bmatrix}
				V_1 TV_1^* && 0 \\
				0 && 0
			\end{bmatrix}
		\mbox{        and      } \mathcal{V}_2 :=\begin{bmatrix}
			0 && 0 \\
			0 && V_2 TV_2^*
		\end{bmatrix}
		\end{equation}
	if $T$ is a projection and
	\begin{equation}
		\mathcal{V}_1 := \begin{bmatrix}
			V_1 TV_1^* && 0 \\
			0 && 1
		\end{bmatrix}
		\mbox{        and      } \mathcal{V}_2 :=\begin{bmatrix}
			1 && 0 \\
			0 && V_2 TV_2^*
		\end{bmatrix}
	\end{equation}
if $T$ is a unitary.
Observe that if we define the matrix
\begin{equation}
	C := \begin{bmatrix}
	0 && V_1V_2^* \\
	V_2V_1^* && 0
	\end{bmatrix}
= \begin{bmatrix}
	0 && 1 \\
	1 && 0
\end{bmatrix}
\cdot \begin{bmatrix}
	V_1V_2^* && 0 \\
	0 && V_2V_1^*
\end{bmatrix}
\end{equation}
then, we have that
\begin{equation}
\mathcal{V}_1 = C \mathcal{V}_2 C^*
\end{equation}
and each entry of $C$ is an isometry of $D^*_\rho(Y, H_Y)^\Gamma$. As consequence of Lemma 4.1.10 of \cite{Anal}, we have that there is a continuous curve of unitary elements connecting $C$ and the matrix
 \begin{equation}
 	\begin{bmatrix}
 		0 && 1 \\
 		1 && 0
 	\end{bmatrix}.
 \end{equation}
Let us define the matrix $\mathcal{W}_2 $ as
\begin{equation}
\mathcal{W}_2 :=\begin{bmatrix}
	V_2 TV_2^* && 0 \\
	0 && 0
\end{bmatrix}
\end{equation}
if $T$ is a projection and
\begin{equation}
	\mathcal{W}_2 :=\begin{bmatrix}
		V_2 TV_2^* && 0 \\
		0 && 1
	\end{bmatrix}
\end{equation}
if $T$ is a unitary. As consequence of Lemma 4.1.10 there is a continuous curve of projection (or unitary) connecting $\mathcal{V}_1$ and $\mathcal{W}_2$.
Then we conclude observing that
\begin{equation}
[A_{V_1}T] = [\mathcal{V}_1] = [\mathcal{W}_2] = [A_{V_2}T].
\end{equation}
	\end{proof}
	\begin{defn}
		Consider two coarse spaces $X$ and $Y$ and let $(H_X, U_X, \rho_X)$ and $(H_Y, U_Y, \rho_Y)$ be a $(X, \Gamma)$-module and a $(Y, \Gamma)$-module. Suppose that $(H_X, U_X, \rho_X)$ and $(H_Y, U_Y, \rho_Y)$ are very ample. Then if $f: X \longrightarrow Y$ is a $\Gamma$-equivariant continuous coarse map, then
		\begin{equation}
		f_\star: K_n(C^*(X)^\Gamma) \longrightarrow K_n(C^*(Y)^\Gamma)
		\end{equation}
		is defined as the morphism induced by $Ad_V$ in K-Theory. We will use the same notation to denote the maps induced by $Ad_V$ between the K-theory groups of $D^*(\cdot)^\Gamma$ and $\frac{D^*(\cdot)^\Gamma}{C^*(\cdot)^\Gamma}$.
	\end{defn}
\begin{rem}
Because of Examples \ref{serna} and \ref{terna}, we know that if $X$ is a coarse metric space or if $X = M \sqcup N$ and it has the coarse structure defined in Example \ref{coarseexem}, then it admits a very ample $(X, \Gamma)$-module. Then we have that if $f$ is a coarse map between coarse spaces endowed with one of these two coarse structures, then there is a well-defined map $f_\star$ in K-theory.
\end{rem}
\begin{rem}
	If we have two continuous coarse maps $f: X \longrightarrow Y$ and $g:Y \longrightarrow Z$, an isometry $V: H_X \longrightarrow H_Y$ which uniformly covers $f$ and an isometry $W: H_Y \longrightarrow H_Z$ then $W \circ V$ uniformly covers $g \circ f$. This fact implies that
	\begin{equation}
		f \longrightarrow f_\star
	\end{equation}
	respect the functorial properties.
\end{rem}
		Consider a coarse map $f:X \longrightarrow Y$ between coarse spaces. As proved in Lemma 6.3.11 of \cite{Anal}, to induce a map between the coarse algebras it is sufficient that the isometry $V$ coarsely covers $f$. Moreover, as we said in Remark \ref{check}, if $f$ and $g: X \longrightarrow Y$ are close maps, then $V$ coarsely covers $f$ if and only if it coarsely covers $g$. 
		\\We already know, because of Example \ref{bru}, that if $f$ and $g:(M,g) \longrightarrow (N,h)$ are two uniform homotopic maps between connected Riemannian manifolds, then they are close. This means that we have
		\begin{equation}\label{ops}
		f \sim_\Gamma g \implies f_\star = g_\star.
		\end{equation}
	where $f_\star, g_{\star}: K_n(C^*(X)^\Gamma) \longrightarrow K_n(C^*(Y)^\Gamma)$.
	\\Then the following Proposition holds.
	\begin{prop}
	Consider $(M,g)$ and $(N,h)$ two Riemannian manifolds. If they are uniform-homotopy equivalent, then
	\begin{equation}
K_n(C^*(M)^\Gamma) \cong K_n(C^*(N)^\Gamma). 
	\end{equation}
	\end{prop}
	\subsection{Roe Index of the signature operator}
		Consider $(M,g)$ a connected, oriented and complete Riemannian manifold.
	\begin{defn}
	Let us denote by $d$ the closed extension in $\mathcal{L}^2(M)$ of exterior derivative operator and by $\tau$ the chirality operator on $(M,g)$. The \textbf{signature operator} $D_M: dom(d) \cap dom(d^*) \subset \mathcal{L}^2(M) \longrightarrow \mathcal{L}^2(M)$ is the operator defined as
	\begin{equation}
	D_M := d + d^* = d - \tau d \tau.
	\end{equation}
if $dim(M)$ is even and as
\begin{equation}
D_M := \tau d + d \tau
\end{equation} 
if $dim(M)$ is odd. 
	\end{defn}
Fix a group $\Gamma$ of isometries on $(M,g)$. Our goal, in this subsection, is to define a class in $K_\star(C^*(M)^\Gamma)$ related to $D_M$. We will call this class the \textit{Roe index} of $D_M$.
\begin{defn}
Let $\chi: \mathbb{R} \longrightarrow \mathbb{R}$ be a smooth map. Then $\chi$ is a \textbf{chopping function} if it is odd, $\lim\limits_{x \to +\infty} \chi = 1$ and $\lim\limits_{x\to -\infty}\chi = -1$.
\end{defn}
Then, since $D_M$ is selfadjoint, the operator
\begin{equation}
\chi(D_M) \in B(\mathcal{L}^2(M))
\end{equation}
is well-defined. In particular, we have that $\chi(D_M)$ is in $D^*(M)^\Gamma$ (see \cite{Roeindex}).
\\Let us denote by $C_0(\mathbb{R})$ the vector space of continuous functions $h: \mathbb{R} \longrightarrow \mathbb{R}$ such that $\lim\limits_{t \rightarrow \pm \infty} h(t) = 0$. 
\\Because of Proposition 3.6 \cite{Roeindex}, if $h$ is a function in $C_0(\mathbb{R})$, then $h(D_M) \in C^*(M)^\Gamma.$
Then, if $\chi_1$ and $\chi_2$ are two chopping functions, we have that $\chi_1 - \chi_2(D_M) \in C^*(M)^\Gamma$.
This means that, given a chopping function $\chi$, the operator
\begin{equation}
\chi(D_M) \in \frac{D^*(M)^\Gamma}{C^*(M)^\Gamma}
\end{equation}
doesn't depend on the choice of $\chi$. Moreover, since $\chi^2 - 1 \in C_0(\mathbb{R}),$ we also have that $\chi(D_M)$ is an involution of $\frac{D^*(M)^\Gamma}{C^*(M)^\Gamma}$. 
\\Let us suppose that $dim(M)$ is odd and consider
\begin{equation}
\frac{1}{2}(\chi (D_{M})+1) \in \frac{D^{\star}(M)^{\Gamma}}{C^{\star}(M)^\Gamma}.
\end{equation}
This is a projection.
\\
\\Let us suppose $dim(M)$ is even. We have that $D_M$ anti-commute with the chirality operator $\tau$. Then, considering the orthogonal splitting $\mathcal{L}^2(M) = \mathcal{L}^2(V_{+1}) \oplus \mathcal{L}^2(V_{-1}),$ we have that $D_M$ can be written as
\begin{equation}
D_M = \begin{bmatrix} 0 && D_{M-} \\
	D_{M+} && 0
	\end{bmatrix}.
\end{equation}
We have that
\begin{equation}
	\chi(D_M) = \begin{bmatrix} 0 && \chi(D_{M})_- \\
		\chi(D_{M})_+ && 0
	\end{bmatrix}.
\end{equation}
Since $\chi(M)$ is a self-adjoint involution, then it is a unitary operator. The same holds also for $\chi(D_{M})_+$ and $\chi(D_{M})_-$.
\\Consider $(H_M, \rho_M, U_M)$ the $(M, \Gamma)$-module defined in Example \ref{terna}: we have that
\begin{equation}
H_M = ... \mathcal{L}^2(V_{-1}) \oplus \mathcal{L}^2(V_{-1}) \oplus \mathcal{L}^2(V_{+1}) \oplus \mathcal{L}^2(V_{-1}) \oplus \mathcal{L}^2(V_{-1}) ...
\end{equation}
and, because of Remark \ref{girl}, we can see $\chi(D_M)_+$ as a bounded operator on $H_M$ defined for each $\{\alpha_j\}$ in $H_M$ as $\beta_i = \alpha_{i+1}$ if $i \neq -1$ and as $\beta_{-1} = \chi(D_{M})_+ (\alpha_0)$. Then we can see $\chi(D_{M})_+$ as an unitary operator in
\begin{equation}
\frac{D^*_{\rho_M}(M, H_M)^\Gamma}{C^*_{\rho_M}(M, H_M)^\Gamma}.
\end{equation}
	\begin{defn}\label{index}
		The \textbf{fundamental class of $D_{M}$} is $[D_{M}] \in K_{n+1}(\frac{D^{\star}(M)^{\Gamma}}{C^{\star}(M)^\Gamma})$ given by
		\begin{equation}
		[D_{M}] := \begin{cases} 
		
		[\frac{1}{2}(\chi (D_{M})+1)] \text{if $n$ is odd,}
		\\ [\chi(D_{M})_+] \text{if $n$ is even.}
		
		\end{cases}
		\end{equation}
	\end{defn}
\begin{rem}
The definition of fundamental class in the even case is well-given: we are considering the $(M, \Gamma)$-module defined in Example \ref{terna} (remember that the K-theory groups of Roe algebra, structure algebra and their quotient don't depend on the Hilbert space or the representation).
\end{rem}
\begin{defn}
	The \textbf{Roe index} of $D_{M}$ is the class
	\begin{equation}
	{Ind_{Roe}(D_{M}) := \delta[D_{M}]}
	\end{equation}
	in $K_{n}(C^{*}(M)^{\Gamma})$, where $\delta$ the connecting homomorphism in the K-Theory sequence.
\end{defn}
\begin{rem}
The definition of Roe index in the even case is coherent with the definition 12.3.5. given in \cite{Anal}. Indeed as the authors show in the proof of Proposition 12.3.7., our fundamental class is the image of the Kasparov class $[D] \in K_p(M)$ under Paschke duality.
\end{rem}
	\section{Smoothing operators}\label{smoothing}
	\begin{defn}
		Consider two complete Riemannian manifolds $(M,g)$ and $(N,h)$. Let us denote by $pr_N$ the projection $pr_N: M\times N \longrightarrow N$ and by $pr_M$ the projection on the first component. Let us define the bundle on $M\times N$ given by
		\begin{equation}
		\Lambda^*(M) \boxtimes \Lambda(N) := pr_M^*(\Lambda^*(M)) \otimes pr_N^*(\Lambda(N)),
		\end{equation}
		where $\Lambda(N)$ is the dual bundle of $\Lambda^*(N)$.
	\end{defn}
	\begin{defn}
		An operator $A: dom(A) \subseteq \mathcal{L}^2(N) \longrightarrow \mathcal{L}^2(M)$ is a \textbf{smoothing operator} if there is a section $K$ of the fiber bundle $\Lambda^*(M) \boxtimes \Lambda(N)$ such that, for $\alpha \in dom(A) \cap \Omega^*(N)$ and for almost all $p \in M$ we have that
		\begin{equation}
		A(\alpha)(p) := \int_N K(p,q)\alpha(q) d\mu_{N}.
		\end{equation}
	\end{defn}
	\begin{rem}
		Given two coordinate charts $\{U, x^s\}$ on $M$ and $\{V, y^l\}$ on $N$, we have that a smooth section of $\Lambda^*(M) \boxtimes \Lambda(N)_{(p,q)}$ is locally given by
		\begin{equation}
		f(p,q)_S^L dx^S \otimes \frac{\partial}{\partial y^L},
		\end{equation}
		where $S = (s_1,... ,s_m)$ and $L = (l_1, ..., l_n)$ are multi-index, $dx^S = dx^{s_1}\wedge ... \wedge dx^{s_m}$, $\frac{\partial}{\partial y^L} = \frac{\partial}{\partial y^{l_1}}\wedge... \wedge \frac{\partial}{\partial y^{l_n}}$ and $f(p,q)_S^L$ is a function in $C^{\infty}(U \times V)$.
	\end{rem}
\begin{rem}\label{fuori}
	We know that for all $p$ in $M$ and for all $q$ in $N$ we have that $\Lambda_p^*(M)$ and $\Lambda_q(N)$ have a norm induced by their metrics (indeed the norm on $\Lambda_q(N)$ is equal to the dual norm of $\Lambda_q^*(N)$). We have that for all $k \in \Lambda^*(M)_p \boxtimes \Lambda(N)_q$ there are some $\beta_{i,p} \in \Lambda^*(M)$ and some $\gamma_{i,q} \in \Lambda(N)$ such that
	\begin{equation}
	k = \sum_i \beta_{i,p}  \otimes \gamma_{i,q}.
	\end{equation}
	Then, imposing for all $\beta_p \in \Lambda_p^*(M)$ and for all $\gamma_{q} \in \Lambda_q(N)$ that
	\begin{equation}
	\vert\beta_p \otimes \gamma_q\vert := \vert\beta_p\vert_{\Lambda_p^*(M)}\cdot\vert\gamma_q\vert_{\Lambda_q(N)},
	\end{equation}
	we can induce a norm on $\Lambda^*(M) \boxtimes \Lambda(N)_{(p,q)}$. Moreover, if the $\beta_{i,p}$ and the $\gamma_{i,q}$ are choosen such that for all $i \neq j$
	\begin{equation}
	\langle \beta_{i,p}, \beta_{j,p} \rangle_{\Lambda_p^*(M)} = \langle \gamma_{i,p}, \gamma_{j,p} \rangle_{\Lambda_q(N)} = 0,
	\end{equation}
	then we have that
	\begin{equation}
	\vert k\vert^2 = \sum_i \vert\beta_{i,p}\vert^2  \cdot \vert\gamma_{i,q}\vert^2.
	\end{equation}
Consider a smooth kernel of an integral operator $K$ and a differential form $\alpha$ in $\Omega^*(N)$. Then, for each point $(p,q)$ we have that
\begin{equation}
	K(p,q)\alpha(q) = \sum_i \gamma_{i,q}(\alpha(q)) \beta_{i,p}
\end{equation}
and so
\begin{equation}
	\begin{split}
		\vert K(p,q)\alpha(q)\vert^2 &= \sum_i \vert\gamma_{i,q}(\alpha(q))\vert^2 \vert\beta_{i,p}\vert^2\\
		&\leq \sum_i \vert\beta_{i,p}\vert^2\vert\gamma_{i,q}\vert^2)\alpha(q)\vert^2 \\
		&\leq (\sum_i \vert\beta_{i,p}\vert^2\vert\gamma_{i,q}\vert^2)\vert\alpha(q)\vert^2 \\
		&= \vert K(p,q)\vert^2 \vert\alpha(q)\vert^2.
	\end{split}
\end{equation}
\end{rem}
	\begin{defn}
		Let us consider two Riemannian manifolds $(N,h)$ and $(M,g)$ and let $A: dom(A) \mathcal{L}^2(N) \longrightarrow \mathcal{L}^2(M)$ be an integral operator with kernel $K$. We have that $A$ is an in integral operator with \textbf{compactly supported} kernel if the support of $K$ as section of $\Lambda^*(M) \boxtimes \Lambda(N)$ is compact.
	\end{defn}
	\begin{rem}\label{suppa}
	Smoothing operators with compactly supported kernels are compact operators.
	\end{rem}
	\begin{defn}
		Consider $N$ and $M$ two Riemannian manifolds. Let $A:\mathcal{L}^2(N) \longrightarrow \mathcal{L}^2(M)$ be a smoothing operator with kernel $K$. We say that $A$ has \textbf{uniformly bounded support} if there is $R \geq 0$ such that for all $q \in N$ and for all $p \in M$ we have that
			\begin{equation}
			diam(supp(K(\cdot, q))) \leq R 
			\end{equation}
			and
			\begin{equation}
			 diam(supp(K(p, \cdot))) \leq R.
			\end{equation}
	\end{defn}
	\begin{prop} \label{smoothbound}
		Let $(N,n)$ and $(M,m)$ be two Riemannian manifolds of bounded geometry. Let $A: dom(A) \mathcal{L}^2(N) \longrightarrow \mathcal{L}^2(M)$ be an integral operator with kernel $K$. If $A$ has uniformly bounded support with constant $Q$ and
		\begin{equation}
		\sup_{(p,q) \in M \times N} \vert K(p,q)\vert < L,
		\end{equation}
		then the operator $A$ is bounded.
	\end{prop}
	\begin{proof}
		Consider a smooth form $\alpha$ in $\mathcal{L}^2(N)$. We have that
		\begin{equation}
			\vert\vert A(\alpha)\vert\vert^2 = \int_M \vert\int_N K(p,q)\alpha(q)d\mu_N\vert^2 d\mu_M.
		\end{equation}
\textbf{Step 1.} Fix $p$ in $M$. Since $A$ is right-uniformly bounded, then there is a $S\geq 0$ such that, for each fixed $p$ in $M$, there is a ball $B_S(q_p) \subset N$ such that
\begin{equation}
K(\tilde{p},q)\alpha(q) = 0
\end{equation}
if $\tilde{p}$ is not in $B_{S}(q_p)$.
\\For each fixed $p$ consider the measure $\mu_p$ on $B_S(q_p)$ defined for each $\mu_N$-measurable set as
\begin{equation}
	\mu_p(B) := \frac{\mu_N(B)}{\mu_N(B_S(q_p))}.
\end{equation}
Observe $B_S(q_p)$ is a probability space. Let us define the map
\begin{equation}
\begin{split}
	F_p: B_S(q_p) &\longrightarrow \Lambda^*_p(M) \\
	q &\longrightarrow \mu_N(B_S(q_p)) \cdot K(p,q)\alpha(q)
\end{split}
\end{equation}
Consider the map 
\begin{equation}
\begin{split}
\phi_p : \Lambda^*_p(M) \cong \mathbb{R}^{2n} &\longrightarrow \mathbb{R}\\
v_p &\longrightarrow \vert v_p \vert_p^2.
\end{split}
\end{equation}
Then we can apply the Multivariate Jensen inequality \cite{Jens} to $\phi$. Considering on $B_S(q_p)$ the measure $\mu_p$ we obtain
\begin{equation}\label{ciccone}
	\begin{split}
\phi_p(\int_{B_S(q_p)} F_p d\mu_p) \leq \int_{B_S(q_p)} \phi_p(F_p) d\mu_p
	\end{split}
\end{equation}
Observe that the left-hand part of (\ref{ciccone}) is
\begin{equation}
\begin{split}
\phi_p(\int_{B_S(q_p)} F_p d\mu_p) &= (\int_{B_S(q_p)} K(p,q)\alpha(q) \frac{\mu_N(B_S(q_p))}{\mu_N(B_S(q_p))} d\mu_N)^2\\
&= \vert \int_{N} K(p,q)\alpha(q) d\mu_N \vert_p^2
\end{split}
\end{equation}
The right-hand of (\ref{ciccone}) is
\begin{equation}
\begin{split}
\int_{B_S(q_p)} \phi_p(F_p) d\mu_p &= \int_{B_S(q_p)} \frac{\mu_N(B_S(q_p))^2}{\mu_N(B_S(q_p))} \vert K(p,q)\alpha(q)\vert^2 d\mu_N \\
&= \mu_N(B_S(q_p)) \int_{B_S(q_p)} \vert K(p,q)\alpha(q)\vert^2 d\mu_N\\
&= \mu_N(B_S(q_p)) \int_{N} \vert K(p,q)\alpha(q)\vert^2 d\mu_N
\end{split}
\end{equation}
Then we have that
\begin{equation}
\vert\int_N K(p,q)\alpha(q) d\mu_N\vert^2 \leq C\cdot \int_N \vert K(p,q)\alpha(q)\vert^2d\mu_N.
\end{equation}
And so we conclude the first step.
\\
\\
\\
\textbf{Step 2.}
Observe that 
		\begin{equation}
		\vert K(p,q)\alpha(q)\vert^2 \leq \vert K(p,q)\vert^2 \vert\alpha(q)\vert^2.
		\end{equation}
We proved in Remark \ref{fuori}. Then we have that
		\begin{equation}
		\begin{split}
		\vert\vert A(\alpha)\vert\vert^2 &= \int_M \vert\int_N K(p,q)\alpha(q)d\mu_N\vert^2 d\mu_M \\
		&\leq C \cdot \int_M \int_N \vert K(p,q)\alpha(q)\vert^2 d\mu_N d\mu_M\\
		&\leq C \cdot  \int_M \int_N \vert K(p,q)\vert^2\cdot \vert\alpha(q)\vert^2 d\mu_N d\mu_M\\
		&\leq C \cdot  \int_N \int_M \vert K(p,q)\vert^2\cdot \vert\alpha(q)\vert^2  d\mu_M d\mu_N\\
		&\leq C \cdot  \int_N (\int_M \vert K(p,q)\vert^2 d\mu_M \vert\alpha(q)\vert^2 d\mu_N.
		\end{split}
		\end{equation}
		Observe that, since $A$ has uniformly bounded support and since $M$ has bounded geometry, then exists a real number $L$
		\begin{equation}
		\mu_M(supp(K(\cdot, q))) \leq Vol(B_L(q)) \leq R
		\end{equation}
		Moreover if $\vert K(p,q)\vert^2 \leq L$, we have that for all $q$ in $N$
		\begin{equation}
		\int_M \vert K(p,q)\vert^2  d\mu_M \leq C_S L^2 \leq R
		\end{equation}
		Then we have that
		\begin{equation}
		\vert\vert A(\alpha)\vert\vert^2 \leq C \cdot R \int_{N}\vert\alpha(q)\vert^2 d\mu_N \leq CR\vert\vert\alpha\vert\vert.
		\end{equation}
	\end{proof}
	\begin{lem}\label{impo}
		Let $A:dom(A) \subseteq \mathcal{L}^2(N) \longrightarrow \mathcal{L}^2(M)$ be a smoothing operator with uniformly bounded support. Then given a multi-index $I$ and given an index $l$, we will denote by $Jl$ the multi-index defined as $Jl :=(j_1,..., j_n,l)$. 
		\\We have that $dA$ is also a smoothing operator and if the kernel of $A$ is locally given by
		\begin{equation}
		K(x,y) = K^I_J(x,y)dx^J\boxtimes \frac{\partial}{\partial y^I}
		\end{equation}
		then $dA$ is an integral operator and its kernel is locally given by
		\begin{equation}
		d_MK^I_S :=_{\vert_{loc}} (\sum\limits_{Jl = S} \frac{\partial}{\partial x^l} K^I_J(x,y))dx^S \boxtimes \frac{\partial}{\partial y^I}
		\end{equation}
	\end{lem}
	\begin{proof}
	It is a direct computation. Calculations can be found in Lemma D.04 \cite{thesis}.
	\end{proof}
	\begin{rem}\label{derivsmooth}
		Observe that the support of the kernel of $dA$ is strictly contained in the support of the kernel of $A$. This means that if $A$ has uniformly bounded support, then also $dA$ has uniformly bounded support.
		\\Moreover if $A$ is a smoothing operator which kernel $K$ has uniformly bounded support and if in normal coordinates there is a uniform bound
		\begin{equation}
		\vert\frac{\partial}{\partial x^j} K^I_J(x,y)\vert \leq C,
		\end{equation}
		then $dA$ is a bounded operator.
	\end{rem}
\section{Operators between manifolds of bounded geometry}
	\subsection{The operator $y$}
	Let us consider two manifolds of bounded geometry $(M,g)$ and $(N,h)$ and let $\delta \leq inj(N)$. Consider a smooth, uniformly proper lipschitz map $f: (M,g) \longrightarrow (N,h)$. Recall that in Lemma 3.3 of \cite{Spes}, we introduced a uniformly proper R.-N.-lipschitz submersion $p_f: (f^*(T^\delta N), g_S) \longrightarrow (N,h)$ such that $p_f(0_{v_p}) =f(p)$. The metric $g_S$ is the Sasaki metric induced by $g$, $f^*\nabla^{LC}_h$ and the metric bundle $f^*h$. Definitions of \textit{R.-N.-lipschitz} and \textit{Sasaki metric} can be found in \cite{Spes} (definitions 2.6 and 3.1). For our ends it is important to know that the pullback of a R.N.-lipschitz map does induce a morphism between the $\mathcal{L}^2$-spaces. Indeed in general a uniform map $f$ doesn't induce a $\mathcal{L}^2$-bounded pullback. In \cite{Spes} we introduced, for each smooth, uniformly proper lipschitz map $f: (M,g) \longrightarrow (N,h)$, the $\mathcal{L}^2$-bounded operator $T_f: \mathcal{L}^2(N) \longrightarrow  \mathcal{L}^2(M)$ defined for each $\alpha$ as $T_f\alpha := \int_{B^\delta} p_f^*\alpha \wedge \omega$, where $B^\delta$ is the fiber of $f^*(T^\delta N)$ and $\omega$ is a specific Thom form of $f^*TN$. This operator plays as $\mathcal{L}^2$-bounded version of the pullback of $f$ and it induces a functor in $\mathcal{L}^2$-cohomology.
	\begin{lem}\label{yformula}
		Let $f:(M,g) \longrightarrow (N,h)$ be a smooth, uniformly proper, lipschitz homotopy equivalence between two complete oriented Riemannian manifolds. Let us consider the bundle\footnote{We use the numbers $1$ and $2$ just to distingush the first and the second summands.}
		\begin{equation}
		f^*(T N)_1 \oplus f^*(T N)_2 \longrightarrow M
		\end{equation}
		and consider $\mathcal{B} \subset f^*(T N)_1 \oplus f^*(T N)_2$ given by
		\begin{equation}
		\mathcal{B} =\{(v_{f_1(p)}, v_{f_2(p)}) \in f^*(T N)_1 \oplus f^*(T N)_2 \vert \vert v_{f_1(p)}\vert \leq \delta, \vert v_{f_2(p)}\vert \leq \delta \}
		\end{equation}
		where $\delta$ is the radius of injectivity of $N$.
		Let us define on $f^*(T N)_1 \oplus f^*(T N)_2$ the Sasaki metric. Then we consider on $\mathcal{B}$ the metric induced by $g_S$.
		\\Consider for $i= 1,2$ the projections
		\begin{equation}
		pr_i : \mathcal{B} \longrightarrow f^*(T^\delta N)_i
		\end{equation}
		and let $p_{f,i}$ be the maps $p_{f,i} = p_f \circ pr_i$. Then $p_{f,i}^*$ induce a $\mathcal{L}^2$-bounded map and there is a $\mathcal{L}^2$-bounded operator $y_0$ such that
		\begin{equation}
		p_{f,1}^* - p_{f,2}^* = d y_0 + y_0d.
		\end{equation}
		for all smooth forms in $\mathcal{L}^2(N)$.
	\end{lem}
	\begin{proof}
		Observe that $p_{f,i}^* = pr_i^* \circ p_f^*.$
	We have that $pr_i$ is a lipschitz submersion with Fiber Volume equal to the volume of a $\delta$-euclidean ball in each point $(p,t)$. Then it is a R.-N.-lipschitz map. Moreover, as we proved in Lemma 3.3 of \cite{Spes}, we have that $f$ is uniformly proper smooth lipschitz map and so $p_f$ is R.-N.-lipschitz. Then since composition of R.-N.-lipschitz maps is R.-N.-lipschitz (Proposition 2.5 of \cite{Spes}), we can conclude that $p_{f,i}^*$ is R.-N.-lipschitz and so its pullback is a $\mathcal{L}^2$-bounded operator.
		\\
		\\In order to prove the second point we need a lipschitz homotopy $H$ between $p_{f,0}$ and $p_{f,1}$ such that $H^*$ is a $\mathcal{L}^2$-continuous map: then the assertion will be proved considering
		\begin{equation}
		y_0 := \int_{0,\mathcal{L}}^1 \circ H^*.
		\end{equation}
		Let us define the map $a: B^\delta_1 \times B^\delta_2 \times [0,1] \longrightarrow B^\delta$ as
		\begin{equation}
		a(t_1, t_2, s) = t_1(1-s) + st_2.
		\end{equation}
	where $B^\delta$ and $B^\delta_i$ are euclidean balls of radius $\delta$ in $\mathbb{R}^n$.
		\\Now, since $B^\delta$ and $[0,1]$ are compact spaces, we have that $a$ is a lipschitz surjective submersion with bounded Fiber Volume. Indeed we know that the Fiber Volume of a submersion is continuous on the image of the submersion. Then $a$ is a R.-N.-lipschitz map.
		\\Let us define the map $A: \mathcal{B} \times [0,1] \longrightarrow f^*T^\delta N$ as
		\begin{equation}
		A(v_{f(p),1}, v_{f(p), 2}, s) := v_{f(p),1}\cdot(1-s) + s\cdot v_{f(p), 2}
		\end{equation}
	and consider the homotopy $H: \mathcal{B} \times [0,1] \longrightarrow N$ defined as
		\begin{equation}
		H(v_{f(p),1}, v_{f(p), 2},s) = p_f \circ A (v_{f(p), 1}, v_{f(p),2},s).
		\end{equation}
	Observe that $A$ is a lipschitz map and the Fiber Volume of $A$ has the same bound of the Fiber Volume of $a$. Then $H$ is an R.N.-lipschitz map because it is composition of R.-N.-lipschitz maps (Proposition 2.5 of \cite{Spes}). This means that $H^*$ is a $\mathcal{L}^2$-bounded operator. So, in particular, since $\int_{0\mathcal{L}}^1$ is a $\mathcal{L}^2$-bounded operator, also the composition
	\begin{equation}
		y_0 := \int_{0\mathcal{L}}^1 \circ H^*
	\end{equation}
is a $\mathcal{L}^2$-bounded operator.
	 \\Observe that $H^*(\Omega_c^*(N)) \subset (\Omega^*(f^*(T^\delta N)_1 \oplus f^*(T^\delta N)_2 \times [0,1]))$. Since $d$ and $H^*$ commute for all smooth forms in $\mathcal{L}^2$, we can apply Corollary 2.3 of \cite{Spes}, and so the formula
		\begin{equation}
		p_{f,1}^*\alpha - p_{f,2}^*\alpha = d y_0\alpha + y_0d\alpha.
		\end{equation}
		holds for every $\alpha$ in $dom(d)$.
	\end{proof}
	\begin{lem}\label{y and Y}
		Consider an oriented manifold $(N,h)$ of bounded geometry. Consider $p_{id}: T^\delta N \longrightarrow N$ the submersion related to the identity map defined in Lemma 3.3 of \cite{Spes}. Consider a Thom form $\omega$ of the bundle $\pi:TN \longrightarrow N$, where $\pi(v_p) = p$, such that $supp(\omega) \subset T^\delta N$. Then for all $q$ in $N$ we have that
		\begin{equation}
		\int_{F_q} \omega = 1
		\end{equation}
		where $F_q$ is the fiber of  $p_{id}$.
	\end{lem}
	\begin{proof}
		For all $q$ in $N$ the fiber $F_q$ is an oriented compact submanifold with boundary. The same also holds for $B^\delta_q$ which is the fiber of the projection $\pi: T^\delta N \longrightarrow N$ defined as $\pi(v_q) := q$.
		\\Let us consider the map $H:T^\delta N \times [0,1] \longrightarrow N$ defined as
		\begin{equation}
		H(v_p,s) = p_{id}(s\cdot v_p).
		\end{equation}
		Since $H$ is a proper submersion, we have that the fiber along $H$ given by $F_{H,q}$ is submanifold of $T^\delta N \times [0,1]$. Its boundary, in particular is
		\begin{equation}
		\partial F_{H,q} = B^\delta_q \times \{0\} \sqcup F_q \times \{1\} \cup A
		\end{equation}
		where $A$ is contained in
		\begin{equation}
			S^\delta N := \{v_{p} \in TN \vert \vert v_p\vert = \delta\}.
		\end{equation}
		Then we have that, if $\omega$ is a Thom form of $TN$ whose support is contained in $T^\delta N$, then
		\begin{equation}\label{nano}
		0 = \int_{F_{H_q}} d \omega = d \int_{F_{H_q}} \omega + \int_{\partial F_{H_q}} \omega.
		\end{equation}
		Observe that $\omega$ is a $k$-form and $dim(F_{H_q})= k+1$. Then the first integral on the right side of \ref{nano} is $0$. Moreover, we obtain that $\omega$ is null on $A$, and so
		\begin{equation}
	\int_A \omega = 0.
		\end{equation}
	 Then, since \ref{nano}, we have that
		\begin{equation}
		0 = \mp \int_{B^\delta} \omega \pm \int_{F_q}\omega.
		\end{equation}
	and we conclude.
	\end{proof}
	\begin{lem}\label{mai2}
		Consider a fiber bundle $p: (M,g) \longrightarrow (N,h)$ where $p$ is a R.N.-lipschitz map. Denote by $p_\star$ the operator of integration along the fibers of $p$. Then if $(p_\star)^*$ is the adjoint of $p_\star$ and $\tau_M$ and $\tau_N$ are the chiral operators of $M$ and $N$, we have that
		\begin{equation}\label{piuma}
		(p_\star)^* = \tau_M \circ p^* \circ \tau_N.
		\end{equation}
		Moreover, if $e_\omega(\alpha) := \alpha \wedge \omega$, then we have that
		\begin{equation}
		(e_{\omega})^*\alpha = (-1)^{deg(\alpha)\cdot n} e_{\omega} \alpha.
		\end{equation}
	\end{lem}
	\begin{proof}It follows since the Projection Formula and a direct computation. See Proposition 1.3.13. of \cite{thesis} for a proof of (\ref{piuma}).\end{proof}
	\begin{prop}
		Let $f:(M,g) \longrightarrow (N,h)$ be a smooth, uniformly proper, lipschitz homotopy equivalence between two manifolds of bounded geometry which maintains the orientation. Then there is a bounded operator $y$ such that\footnote{We have that if $X$ and $Y$ are two Riemannian manifolds and $A: \mathcal{L}^2(X) \longrightarrow \mathcal{L}^2(Y)$ is a $\mathcal{L}^2$-buonded operator, then $A^\dagger = \tau_{X} \circ A^* \circ \tau_Y$.}
		\begin{itemize}
			\item $y(dom(d_{min})) \subseteq dom(d_{min})$,
			\item on $dom(d_{min})$
			\begin{equation}\label{smo}
				1 - T^\dagger_fT_f = d y + y d ,
			\end{equation}
		\item $y^\dagger = y$.
		\end{itemize}
	\end{prop}
	\begin{proof}
		Let us consider two smooth $\mathcal{L}^2$-forms $\alpha$ and $\beta$ on $N$. First of all we will consider the case in which $\alpha$ and $\beta$ are both $\mathcal{L}^2$-smooth $j$-forms for some natural number $j$. Let us consider the fiber bundle $\mathcal{B} := f^*(T^\delta N)_1 \oplus f^*(T^\delta N)_2$ with the metric defined in Lemma \ref{yformula}. We denote, moreover with $B^\delta_1$ and with $B^\delta_2$ the fibers of $f^*(T^\delta N)_1$ and $f^*(T^\delta N)_2$.
		\\ Using that $(T_f^\dagger)^* = (\tau T_f^* \tau)^*$ and that $\tau$ is self-adjoint, we obtain
		\begin{equation}
		\begin{split}
		\langle T^\dagger_fT_f \alpha, \beta \rangle &=  \langle T_f \alpha, \tau T_f \tau \beta \rangle \\
		&= \int_{M} (\int_{B^\delta_2}p_{f,2}^*\alpha \wedge \omega_2) \wedge (\int_{B^\delta_1}p_{f,1}^*\tau \beta \wedge \omega_1) \\
		&= \int_{f^*(T^\delta N)_1} (\int_{B^\delta_2}p_{f,2}^*\alpha \wedge \omega_2) \wedge p_{f,1}^*\tau \beta \wedge \omega_1 \\
		&= (-1)^{(n + j)j} \int_{f^*(T^\delta N)_1} p_{f,1}^*\tau \beta \wedge \omega_1 \wedge (\int_{B^\delta_2}p_{f,2}^*\alpha \wedge \omega_2)  \\
		&= (-1)^{(n + j)j} \int_{\mathcal{B}} p_{f,1}^*\tau \beta \wedge \omega_1 \wedge p_{f,2}^*\alpha \wedge \omega_2  \\
		&= (-1)^{(n + j)j}(-1)^{(n + j)j} \int_{\mathcal{B}} p_{f,2}^*\alpha \wedge p_{f,1}^*\tau \beta \wedge \omega_1 \wedge \omega_2  \\
		&= \int_{\mathcal{B}} p_{f,2}^*\alpha \wedge p_{f,1}^*\tau \beta \wedge \omega_1  \wedge \omega_2
		\end{split}
		\end{equation}
	Consider the identity we proved in Lemma \ref{yformula}
		\begin{equation}
		p_{f,2}^* = p_{f,1}^* + dy_0 + y_0 d.
		\end{equation}
		We obtain
		\begin{equation}
		\begin{split}
		\langle T^\dagger_fT_f \alpha, \beta \rangle &= \int_{f^*(TN)_1}p_{f_1}^*(\alpha \wedge \tau \beta) \wedge \omega_1 \\
		&+ \int_{\mathcal{B}} (dy_0 + y_0 d) \alpha \wedge \omega_1 \wedge p_{f,2}^*\tau \beta \wedge \omega_2.
		\end{split}
		\end{equation}
		Since $\alpha\wedge \tau\beta \wedge \omega_1$ is a top-degree form, then its closed and, in particular, the first integral can be written as
		\begin{equation}
		\begin{split}
		\int_{f^*(TN)_1}p_{f_1}^*(\alpha \wedge \tau \beta) \wedge \omega_1 &= \int_{f^*(TN)_1} (f,id_{B^k})^* \circ p_{id}^*(\alpha \wedge \tau \beta) \wedge \omega_1 \\
		&= \int_{TN} p_{id}^*(\alpha \wedge \tau \beta) \wedge \omega_1 \\
		&= \int_N (\int_F \omega_1) \alpha \wedge \tau \beta \\
		&= \int_N \alpha \wedge \tau \beta \\
		&= \langle 1(\alpha), \beta \rangle,
		\end{split}
		\end{equation}
	where $F$ is the fiber of $p_{id}$.
		\\Consider the second integral: let us denote by $pr_1: \mathcal{B} \longrightarrow f^*TN_2$ the projection $(v_{f_1(p)}, v_{f_2(p)}) \rightarrow v_{f_2(p)}$. We have that
		\begin{equation}
		\begin{split}
		&\int_{\mathcal{B}} (dy_0 + y_0 d) \alpha \wedge p_{f,2}^*\tau \beta \wedge \omega_1 \wedge \omega_2\\
		&= (-1)^{j(n-j)} \int_{\mathcal{B}} p_{f,2}^*\tau \beta \wedge (dy_0 + y_0 d)\alpha \wedge \omega_1 \wedge \omega_2 \\
		&= (-1)^{j(n-j)} \int_N \tau \beta \wedge \{ \int_{F_2}[\int_{B^\delta_1}(dy_0 + y_0 d)\alpha \wedge \omega_1 ]\wedge \omega_2 \} \\
		&= \langle (d\circ Y \alpha + Y \circ d) \alpha, \beta \rangle
		\end{split}
		\end{equation}
		where
		\begin{equation}
		\begin{split}
		Y :&=  p_{f,2\star} \circ e_{\omega_2} \circ pr_{1,\star} \circ e_{\omega_1} \circ y_0 \\
		&=  p_{f,2\star} \circ e_{\omega_2} \circ pr_{1,\star} \circ e_{\omega_1} \circ \int_{0, \mathcal{L}}^1 \circ H^*.
		\end{split}
		\end{equation}
Let us assume that $\alpha$ is a $j$-form and $\beta$ is a $r$-form with $j \neq r$, we have that
		\begin{equation}
		\langle T_f^\dagger T_f \alpha, \beta \rangle = \langle (-1 + dY + Yd)(\alpha), \beta \rangle = 0.
		\end{equation}
		Indeed $[-1 + dY + Yd]\alpha$ is a $j$-form if $\alpha$ is a $j$-form.
\\Since proposition 3.2 of \cite{Spes}, we have that $pr_{1,\star}$ is a $\mathcal{L}^2$-bounded operator. Moreover we also have that $p_{f,2\star}$, $e_{\omega_2}$, $e_{\omega_1}$ are $\mathcal{L}^2$-bounded. Then, using Lemma \ref{yformula}, we conclude that $Y$ is a $\mathcal{L}^2$-bounded operator.
		\\Consider $y := \frac{Y + Y^\dagger}{2}$. Then we have that $y$ is an $\mathcal{L}^2$-bounded operator such that $y^\dagger = y$. Observe that since $Y$, on $\Omega^*(N) \cap \mathcal{L}^2(N)$ satisfies the equality (\ref{smo}), then also $y$ satisfies it.
		We still have to prove that $y(dom(d_{min})) \subset dom(d_{min})$: using this fact the proof of the equality (\ref{smo}) for every $\alpha$ in $dom(d_{min})$ will be immediate. Since Proposition 2.2. of \cite{Spes}, in order to to prove $y(dom(d_{min})) \subset dom(d_{min})$, it is sufficient to prove that $y(\Omega_c^*(N)) \subseteq \Omega_c^*(N)$.
		Observe that if $\beta$ is in $\Omega^*_c(N)$ the $Y\beta$ is in $\Omega^*_c(N)$ since $f$ is a proper map. Moreover we also have that $Y^\dagger \beta$ is in $\Omega^*_c(N)$, indeed we have that
		\begin{equation}
		Y^\dagger \beta = (-1)^{n} H_\star \circ pr_{\mathcal{B}}^* \circ e_{\omega_1} \circ pr_{2}^* \circ e_{\omega_2} \circ p_{f,2}^* \beta.
		\end{equation}
		Then we obtain that $y(\Omega_c^*(N)) \subseteq \Omega_c^*(N)$ and this concludes the proof.
	\end{proof}
	\begin{lem}
		Let $(M,g)$ be an orientable Riemannian manifold and let $\Gamma$ be a group of isometries which preserve the orientation. Let $A$ un operator on $\mathcal{L}^2(M)$. If $A$ commute with $\Gamma$, then
		\begin{equation}
		A^* \Gamma = \Gamma A^* \mbox{         and            } A^\dagger \Gamma = \Gamma A^\dagger.
		\end{equation}
	\end{lem}
	\begin{cor}
		If there exists a group $\Gamma$ of isometries which commute with $f$, then also $y$ commute with $\gamma$ for all $\gamma$ in $\Gamma$.
	\end{cor}
	\subsection{$T_f$ as smoothing operator}
	\begin{prop}\label{Tsmooth}
		Consider $f:(M,g) \longrightarrow (N,h)$ a smooth and lipschitz uniform homotopy equivalence between manifolds of bounded geometry. Then the operator $T_f$ is a smoothing operator.
	\end{prop}
	\begin{proof}
		We divide the proof in two steps: in the first one we will show that for all smooth forms $\alpha$ in $\mathcal{L}^2(N)$ we have that
		\begin{equation}
		T_f\alpha(p) = \int_N K(p,q) \alpha(q) d\mu_N
		\end{equation}
		with $K(p,q) \in \Lambda^*(M) \boxtimes \Lambda(N)_{(p,q)}$. In the second one we will prove that the section $K$ is smooth.
		\\
		\\ \textbf{1.} The proof of the first step is similar to the proof given by Vito Zenobi in his Ph.D. thesis \cite{vito}.
		\\Let us denote the following maps as follows
		\begin{itemize}
			\item $p_f:f^*T^\delta N \longrightarrow N$ the submersion related to $f$,
			\item $t_f:f^*T^\delta N \longrightarrow M\times N$ is the map $t:=(\pi, p_f)$,
			\item $\pi:f^*TN \longrightarrow M$ is the projection of the bundle $f^*TN$,
			\item $pr_M$ and $pr_N$ the projections of $M\times N$ over $M$ and $N$.
		\end{itemize}
	   We proved, in Lemma 4.1 of \cite{Spes},  that $t_f$ is a diffeomorphism on its image, so, in particular, it is a submersion and the integration along the fibers of $t_f$ is the pullback along $t_f^{-1}: im(t_f) \longrightarrow f^*T^\delta N$.
	   \\Observe that $p_f = pr_N \circ t_f$ and $\pi = pr_M \circ t_f.$
	   Consider a smooth form $\alpha$ in $\mathcal{L}^2(N)$. Applying the Projection Formula, we obtain that
	   \begin{equation}
	   	\begin{split}
	   	T_f \alpha &= \pi_\star \circ e_\omega \circ p_f^* \alpha \\
	   	&= pr_{M, \star} \circ t_{f, \star} \circ e_\omega \circ t_f^* \circ pr_N^* \alpha \\
	   	&= pr_{M, \star} \circ e_{t_{f, \star} \omega} \circ pr_N^* \alpha,
	   	\end{split}
	   \end{equation}
	   where $t_{f, \star} \omega$ is the form $t_f^{-1}\omega$ on $im(t_f) \subset M \times N$ and $0$ in $im(t_f)^c$. Since $supp(\omega)$ is contained in $f^*T^{\delta_1}N$ where $\delta_1 < \delta$, then $t_{f, \star} \omega$ is a well-defined, smooth form.
	   \\Observe that the Hodge star operator $\star_N$ on $N$ induces a bundle endomorphism on $\Lambda^*(M \times N)$ imposing that, for each couple of orthonormal frames $\{W^i\}$ of $\Lambda^*M$ and $\{\epsilon^j\}$ of $\Lambda^*N$, we have
   \begin{equation}
   	\star_N (pr_M^* W^I \wedge pr_N^* \epsilon ^J (p,q)) := pr_M^* W^I \wedge (\star_N \epsilon ^J) (p,q)
   	\end{equation}
   Consider $\Lambda^{\star, k}(M \times N)$ the subbundle of $\Lambda^*(M \times N)$ given by the forms which are $k$-forms with respect to $N$. Observe that
   \begin{equation}
   	\Lambda^*(M \times N) = \bigoplus\limits_{k \in \mathbb{N}}\Lambda^{\star, k}(M \times N).
   \end{equation}
   Let us denote by $n$ the dimension of $N$. We can define the bundle morphism
	   \begin{equation}
	   B: \Lambda^*(M \times N) \longrightarrow \Lambda^{\star, 0}(M \times N) = pr_M^*(\Lambda^* M)
	   \end{equation}
   which is 
   \begin{equation}
   	B\alpha_{(p,q)} := \star_N \alpha_{(p,q)} 
   \end{equation}
   if $\alpha_{(p,q)} \in \Lambda^{\star, n}(M \times N)_{(p,q)}$ and it is zero if $\alpha_{(p,q)}$ is in $\Lambda^{\star, q}(M \times N)_{(p,q)}$ where $q \neq n$.
   \\Then we have that for each
   \begin{equation}\label{piru}
   	\begin{split}
   	T_f \alpha(p) &= pr_{M, \star} \circ e_{t_{f, \star} \omega} \circ pr_N^* \alpha(p) \\
   &= \int_N B \circ e_{t_{f, \star} \omega} (\alpha)(p,q) d\mu_N
   	\end{split}
   \end{equation}
Let us define, for each $(p,q)$ in $M \times N$ the operators
\begin{equation}
	\begin{split}
E_{\omega, p,q}: pr_N^*(\Lambda^*N)_{(p,q)} &\longrightarrow  \Lambda^*(M \times N)_{(p,q)} \\
\beta_{(p,q)} &\longrightarrow \beta_{(p,q)} \wedge t_{f, \star} \omega(p,q)
	\end{split}
\end{equation}
and
\begin{equation}
	\begin{split}
		B_{p,q}: \Lambda^*(M \times N)_{(p,q)} &\longrightarrow pr_M^*(\Lambda^*M)_{(p,q)}  \\
		\gamma_{(p,q)} &\longrightarrow B\gamma_{(p,q)}.
	\end{split}
\end{equation}
Then the equation (\ref{piru}) can be read as
\begin{equation}
T_f \alpha(p) = \int_N (B_{p,q} \circ E_{\omega, p,q})pr_N^*\alpha(q) d\mu_N =
\int_N K(p,q)\alpha(q) d\mu_N,
\end{equation}	   
where
\begin{equation}
K(p,q) := B_{p,q} \circ E_{\omega, p,q}:pr_N^*(\Lambda^*N)_{(p,q)} \longrightarrow pr_M^*(\Lambda^*M)_{(p,q)}
\end{equation}	   	   
is, for each $(p,q)$ in $M \times N$ an element of $\Lambda^*M \boxtimes \Lambda TN_{(p,q)}$.	   
\\
		\\\textbf{2.} 
		In order to prove that $K$ is a smooth section of $\Lambda^*M \boxtimes \Lambda N$ it is sufficient to show that, in local coordinates $\{x^i, y^j\}$, the kernel $K$ has the form
		\begin{equation}
		K(x,y) = K^J_I(x,y)dx^I \boxtimes \frac{\partial}{\partial y^J}
		\end{equation}
	where $K^J_I$ are smooth functions.
	\\Let us fix some normal coordinates $\{x,y\}$ on $M \times N$. We have that
		\begin{equation}
		t_{f, \star} \omega (x,y) = \beta_{IR}(x,y)dx^I \wedge dy^R
		\end{equation}
		for some smooth functions $\beta_{IR}$ and that $\alpha$ has the form 
		\begin{equation}
		\alpha(y) = \alpha_{J}(y)dy^J.
		\end{equation}
	 Let us denote by $\frac{\partial}{\partial x^I} = \frac{\partial}{\partial x^{i_1}} \wedge ... \wedge \frac{\partial}{\partial x^{i_k}}$, the dual of $dx^I$.
		Let $J^c$ be the multindex such that, up to double switches, $(J, J^c)$ is the index $(1, ..., n)$. We have that
		\begin{equation}
		\begin{split}
		K^J_I(x,y) &= K(\frac{\partial}{\partial x^I}, dy^J) \\
		&= [B_{x,y}(\beta_{LR}(x,y)dx^L \wedge dy^R \wedge dy^J)](\frac{\partial}{\partial x^I})\\
		&= [\beta_{LJ^c}(x,y)dx^L](\frac{\partial}{\partial x^I})\\
		&= \beta_{IJ^c}(x,y).
		\end{split}
		\end{equation}
	    This fact implies that if $K$ in local coordinates has the form 
	    \begin{equation}
	    K(x,y) = \beta_{IJ^c}(x,y) dx^I \boxtimes \frac{\partial}{\partial y^J}
	    \end{equation}
	    and so the kernel $K$ is smooth.
	\end{proof}
	\subsection{Some operators in $C^*_f(M \sqcup N)^\Gamma$}
	In this section we consider operators between $\mathcal{L}^2(M)$ and $\mathcal{L}^2(N)$ as operators in $\mathcal{L}^2(M \sqcup N)$ extended by zero outside their original domains.
	\begin{prop}
		The operator $T_f$ is in $C_f^*(M \sqcup N)^\Gamma$.
	\end{prop}
	\begin{proof}
		Observe that
		\begin{equation}
		supp(T_f) \subseteq \{(x,x') \in M \times N \vert B^\delta_x \cap p_f^{-1}(x') \neq \emptyset \}
		\end{equation}
	where $B^\delta_x$ is the fiber of $\pi: f^*TN \longrightarrow M$.
		\\So, if $\psi$ and $\phi$ are two compactly supported functions on $N$ and $M$ respectively we have that
		\begin{equation}
		supp(\psi T_f \phi) \subseteq \pi^{-1}(supp(\psi)) \cap p_f^{-1}(supp(\phi)).
		\end{equation}
		And so
		\begin{equation}
		d(\pi^{-1}(supp(\psi)), p_f^{-1}(supp(\phi))) > 0 \implies \psi T_f \phi = 0.
		\end{equation}
		\\
		Now, if $t \in p_f(\pi^{-1}(supp(\psi)))$ then $t = p_f(s,q)$ for some $(s,q)$. Then we have
		\begin{equation}
		d(t, f(supp(\psi))) \leq d(p_f(s,q), p_f(s, 0)) \leq C_{p_{f}}d((s,q), (s, 0)) \leq C_{p_{f}}.
		\end{equation} 
		Let $R > C_{p_{f}}$. We have that
		\begin{equation}
		\begin{split}
		&d(f(supp(\psi)), supp(\phi)) > R \implies d(p_f(\pi^{-1}(supp(\psi))), supp(\phi)) > 0 \implies \\
		&d(\pi^{-1}(supp(\psi)), p_f^{-1}supp(\phi)) > 0 \implies \psi T_f \phi = 0.
		\end{split}
		\end{equation}
		It means that $T_f$ has propagation less or equal to $C_{p_{f}}$. Moreover if $\phi$ has compact support then $T_f\phi$ and $\phi T_f$ are integral operator with smooth, compactly supported kernel. Then, in particular they are compact.
		\\Since the $\Gamma$-invariance of $T_f$ it follows that $T_f$ is in $C_f^*(M\sqcup N)^\Gamma$.
	\end{proof}
	\begin{cor}
		Since $T_f$ is in $C_f^*(M\sqcup N)^\Gamma$, then also $T_f^\dagger$, $\tau T_f\tau$ and $\tau T_f^{\dagger}\tau$ are in $C_f^*(M\sqcup N)^\Gamma$.
	\end{cor}
	\begin{proof}
		It is sufficient to observe that $\tau$ is in $D_f^*(M \sqcup N)^\Gamma$ and that, since $C_f^*(M\sqcup N)^\Gamma$ is a $C^*$-algebra, if $T_f$ is in $C_f^*(M\sqcup N)^\Gamma$ then $T_f^*$ is in $C_f^*(M\sqcup N)^\Gamma$.
	\end{proof}
\begin{prop}
	The operator $y$ is a $\Gamma$-equivariant operator with finite propagation. Moreover the operators $T_fy$, $yT_f$, $T_f^\dagger y$ $yT_f^\dagger$ are all operators in $C_f^*(M \sqcup N)^\Gamma$. 
\end{prop}
\begin{proof}
	Let us consider the operator $Y$ in Lemma \ref{y and Y}. If we show that $Y$ is an operator with finite propagation, then one can  check that the same holds for $Y^\dagger$ and for $y:= \frac{Y + Y^\dagger}{2}$.
	\\
	\\First, we have to observe that
	\begin{equation}
	Y:= p_{f,2\star} \circ e_{\omega_2} \circ pr_{1,\star} \circ e_{\omega_1} \circ \int_{0, \mathcal{L}}^1 \circ H^*.
	\end{equation}
	where $H$ is the homotopy defined in Lemma \ref{yformula}. Then its pullbcak is given by
	\begin{equation}
	H^* = A^* \circ pr_1^* \circ p_{f,2}^*.
	\end{equation}
	Let $\alpha$ be a compactly supported differential forms on $N$. 
	\\Since the formula (1.170) in Corollary 1.5.2 in \cite{thesis}, we have that
	\begin{equation}
	supp(p_{f,2}^*\alpha) = p_{f,2}^{-1}(supp(\alpha)) \subseteq \pi_2^{-1}(f^{-1}(B_1(supp(\alpha)))),
	\end{equation}
where $\pi_2: f^*(T^\delta N)_2 \longrightarrow M$ is the projection of the fiber bundle and where $B_1(supp(\alpha))$ is the $1$-neighborhood of $supp(\alpha)$.
	\\This means that
	\begin{equation}
	supp(pr_{1,\star} \circ e_{\omega_2} \circ e_{\omega_1} \circ \int_{0\mathcal{L}} \circ H^*(\alpha)) \subseteq \pi_2^{-1}(f^{-1}(B_1(supp(\alpha)))).
	\end{equation}
	and, in particular
	\begin{equation}
	supp(Y(\alpha)) \subseteq p_{f,2}(\pi_2^{-1}(f^{-1}(B_1(supp(\alpha)))))
	\end{equation}
	Moreover, if $(x,t)$ is in $\pi_2^{-1}(f^{-1}(B_1(supp(\alpha))))$, then
	\begin{equation}
	d(p_f(x,t), B_1(supp(\alpha)) \leq C_{p_f}
	\end{equation}
	where $C_{p_f}$ is the lipschitz costant of $p_f$. Indeed it is sufficient to observe that
	\begin{equation}
	p_f(x,0) = f(x) \subseteq B_1(supp(\alpha))
	\end{equation}
	and
	\begin{equation}
	d(p_f(x,1), p_f(x,0)) \leq C_{p_f}d((x,0),(x,1)) = C_{p_f}.
	\end{equation}
	Then we have that if $\phi$ and $\psi$ are two functions on $N$ with
	\begin{equation}
	d(supp(\phi), supp(\psi)) > C_{p_f} + 1 \implies \phi Y \psi = 0.
	\end{equation}
	\\
	\\Now in order to conclude the proof it is sufficient to show that $T_fY$, $YT_f$, $T_f^\dagger Y$, $Y T_f^\dagger$ are all operator in $C_f^*(M \sqcup N)^\Gamma$. We will start by studying $T_fY$.
	\\We know that $Y$ and $T_f$ are bounded operator which have both finite propagation, so also $T_fY$ is a bounded operator with finite propagation.
	\\Moreover if $g$ is a function on $M \sqcup N$ with compact support than we have that $gT_f$ is a compact operator since $T_f$ is in $C_f^*(M \sqcup N)^\Gamma$ and son $gT_fY$ is compact.
	\\Now let us consider $T_fYg$: now we know that for all $\alpha$ we have that
	\begin{equation}
	supp(Yg(\alpha)) \subseteq B_1(supp(g)).
	\end{equation}
	This means that, if $\phi$ is a compactly supported function such that $\phi \cong 1$ on $B_1(supp(g))$, then we have that
	\begin{equation}
	T_fYg = T_f \phi Y g.
	\end{equation}
	Observe that $T_f\phi$ is a compact operator. Then also holds for $T_fYg$. This means that $T_fY$ is in $C_f(M \sqcup N)^\Gamma$. Exactly in the same way one can check that $YT_f$, $T_f^\dagger Y$, $Y T_f^\dagger$ are operator in $C_f^*(M \sqcup N)^\Gamma$.
\end{proof}
\begin{rem}\label{FVdit}
Let us consider the kernel $K$ of $T_f$. Let us fix some local normal coordinates $\{x,y\}$ on $M \times N$. Since Proposition \ref{Tsmooth}, we know that outside $im(t_f)$ the kernel is identically $0$. Morever, inside $im(t_f)$, we have that, the kernel is locally given by
\begin{equation}
K_J^I(x,y) = \beta_{IJ^c}(x,y), 
\end{equation} 
where the functions $\beta_{IJ^c}(x,y)$ are the components of the pullback of the Thom form $\omega$ along the map $t_{f}^{-1}: im(t_f) \subset M \times N \longrightarrow f^*TN$.
\\
\\Let us recall that, respect to some fibered coordinates $\{x^i, y^j\}$ on $f^*T^\delta N$, the components of the Thom form $\omega$ are algebraic combinations of pullback of components of the metric $h$ on $N$ and of Christoffell symbols of the Levi-Civita connection $\nabla^E$ along the map $f$ and derivatives of $f$ (see subsection 4.2 of \cite{Spes}). This means that if $f$ is a $C^k_b$-map the local components of $\omega$ and their derivatives of order $k$ are uniformly bounded only .
\\Moreover we also know, because Lemma 3.3 \cite{Spes}, that if $f$ is a $C^k_b$-map for each $k$ in $\mathbb{N}$, then also $p_f$ and, in particular, $t_f$ have uniformly bounded derivatives of each order.
\\
\\Let us suppose that $t_f$ has bounded derivatives of order $k$. Then we know that $t_f$ is a diffeomorphism with its image. This fact, using also the inverse function Theorem, implies that $t_{f}^{-1}$ (up to consider $\delta \leq inj(N)$ small enough) has bounded derivatives of order $0, 1, ... k$.
\\
\\The bounds on the components of $\omega$ and on the derivatives of $t_f^{-1}$ imply that if $f$ is a $C^k_b$ map, then the components of the kernel $K$ of $T_f$ in normal coordinates have bounded derivatives of order $0, 1, ... k$.
\end{rem}
\begin{prop}\label{dT_fboun}
	Let $(M,g)$ and $(N,h)$ be two Riemannian manifolds of bounded geometry. Let $f:M \longrightarrow N$ be a smooth and lipschitz uniform homotopy equivalence. Then we have that $T_f$ and $dT_f =T_f d$ have uniformly bounded support (definition in Section \ref{smoothing}). Moreover if $f$ is a $C^2_{b}$-map, then $dT_f$ is a bounded operator.
\end{prop}
\begin{proof}
	We will first show that $T_f$ has uniformly bounded supported kernel $K$. Since $supp(d_M K) \subseteq supp(K)$ it will follow immediately that also $dT_f$ has uniformly bounded supported kernel. 
	\\
	\\We know, since Proposition \ref{Tsmooth} that $K(p,q) = 0$ if $(p,q)$ is not in $im(t_f)$, where $t_f = (\pi, p_f)$. \\Fix $p$, a point of $M$. Observe that if $q$ is not in $p_f(B^\delta_p)$, then $(p,q) \notin im(t_f)$ and so $K(p,q) = 0$. We know, since $p_f$ is lipschitz and $diam(B^\delta) = 2\delta$, that
	\begin{equation}
	diam(supp(K(p, \cdot)) \leq  2\delta C_{p_f}.
	\end{equation}
	Fix now $q$ in $N$. We have that, because of formula (1.170) in Corollary 1.5.2 of \cite{thesis}, that
	\begin{equation}
	\pi(p_f^{-1}(q)) \subseteq B_1(f^{-1}(q))
	\end{equation}
	and if $p$ is not in $B_1(f^{-1}(q))$ then $K(p,q) = 0$. Then, since $f$ is a uniform homotopy equivalence, then in particular we have that $f$ is a uniformly proper map. This means that
	\begin{equation}
	diam(f^{-1}(A)) \leq \alpha(diam(A))
	\end{equation}
	for a function $\alpha: [0, +\infty] \longrightarrow [0, +\infty)$. Then we have that
	\begin{equation}
	diam(supp(K(\cdot, q)) \leq  \alpha(0) +1.
	\end{equation}
	In order to conclude the proof we want to apply Proposition \ref{smoothbound}. In order to apply this Proposition we have to show that
	\begin{equation}\label{true}
		\vert\frac{\partial}{\partial x^j}K^I_J(x,y)\vert \leq C.
	\end{equation}
	Because of Remark \ref{FVdit} and Remark \ref{derivsmooth} in Section \ref{smoothing} this fact imediatly follows if $f$ is a $C^2_b$-map.
\end{proof}
\begin{cor}\label{corn}
	The operators $d_{M}T_f = T_fd_{N}$, $T_f^\dagger d_{M} = d_{N} T_f^\dagger$, $T_fyd_{N}$, $d_{M}T_fy$, $y T_f^\dagger d_{M}$ $d_NyT_f^\dagger$ are all operators in $C^*_f(M \sqcup N)^\Gamma$.
\end{cor}
\begin{proof}
	Since $d_{M}T_f = T_fd_{N}$ is a bounded integral operator with $supp(d_{M} T_f) \subseteq supp(T_f)$ then $d_{M}T_f$ has finite propagation. Moreover, since it has finite propagation we have that for all $g$ in $C_c(M \sqcup N)$ then $d_{M}T_fg$ and $gdT_f$ are smoothing operator with compact support and so are compact operators (Remark \ref{suppa}).
	\\
	\\Observe that
	\begin{equation}
	T_f^\dagger d_{M} = -T_f^\dagger d_{M}^\dagger = -(d_{M} T_f)^\dagger
	\end{equation}
	and since $d_{M}T_f$ is in $C^*(M\sqcup N)^\Gamma$ we have that the same holds for $T_f^\dagger d_{M} = d_{N} T_f^\dagger$.
	\\Finally Observe that
	\begin{equation}
	\begin{split}
	T_fyd_{N} &= T_f(d_{N}y - 1 - T_f^\dagger T_f) \\
	&=(T_fd_{N})y - T_f - T_f T_f^\dagger T_f
	\end{split}
	\end{equation}
	which is a combination of operators in $C_f^*(M\sqcup N)^\Gamma$. All the other combinations are in $C_f^*(M \sqcup N)^\Gamma$ for the same argument. 
\end{proof}
	\section{Uniform homotopy invariance}\label{wai}
	\subsection{The perturbed signature operator}
	Consider two oriented, connected, Riemannian manifolds $(M,g)$ and $(N,h)$ and fix a lipschitz homotopy equivalence $f: (M,g) \longrightarrow (N,h)$. Since $f$ is a lipschitz homotopy equivalence, we have that $dim(M) = dim(N)$.
	Consider $\mathcal{L}^2(M \bigsqcup N) \cong \mathcal{L}^2(N) \oplus \mathcal{L}^2(M)$. 
	\begin{defn}
		Let us define, for each $z$ in $\mathbb{N}$, $p(z) = 0$ in $z$ is even and $p(z) = 1$ otherwise. We denote by $d_{M \sqcup N}$ the operator
		\begin{equation}
		d_{M \sqcup N} := i^{p(dim(M))}\begin{bmatrix}
		d_N && 0 \\
		0 && -d_M
		\end{bmatrix}.
		\end{equation}
		if the dimensions are odd. Moreover we define the \textbf{Signature operator} as
		\begin{equation}
		D_{M \sqcup N} := \begin{bmatrix}
		D_N && 0 \\
		0 && -D_M
		\end{bmatrix}
		\end{equation}
	\end{defn}
	Let $\gamma: \mathcal{L}^2(M\sqcup N) \longrightarrow \mathcal{L}^2(M\sqcup N)$ be the operator defined for all $\alpha$ as $\gamma (\alpha) := (-1)^{\vert\alpha\vert}$. We have that $\gamma ^{\dagger} = -\gamma$.
	\\Let us define $\begin{bmatrix}
	1 && 0 \\
	\beta T_f && 1
	\end{bmatrix}.$
	We have that $R_\beta$ is $\mathcal{L}^2$-invertible. Let $\tau$ be the chirality operator $\tau :=
	\begin{bmatrix}
	\tau_N && 0 \\
	0 && -\tau_M
	\end{bmatrix}.$
	The operator $\gamma$ commutes with $\tau$ and with $d_{M\sqcup N}$.
	\\Consider a real number $\alpha$: we define
	\begin{equation}
	L_{\alpha, \beta} := \begin{bmatrix}
	1- T^{\dagger}_fT_f && \beta (\gamma + \alpha y)T_f^{\dagger} \\
	\beta T_f(-\gamma -\alpha y ) && 1
	\end{bmatrix}.
	\end{equation}
	Moreover, if the dimension of $M$ is even, we will define the operator $\delta_{\alpha}$ as follow
	\begin{equation}
	\delta_{\alpha} :=
	\begin{bmatrix}
	d_N && \alpha iT^{\dagger}_f\gamma \\
	0 && -d_M
	\end{bmatrix},
	\end{equation}
	in the other case we have
	\begin{equation}
	\delta_{\alpha} :=
	\begin{bmatrix}
	i\cdot d_N && -\alpha T^{\dagger}_f \gamma \\
	0 && -i \cdot d_M
	\end{bmatrix},
	\end{equation}
We have that $\delta_{\alpha}^2 = 0$. Since $d_{M \sqcup N}$ is a closed operator and $T^{\dagger}$ is bounded, then $\delta_\alpha$ is a closed operator with the same domain of $d_{M \sqcup N}$.
	\\Observe that $L_{\alpha, \beta} d_{M \sqcup N} = d_{M \sqcup N} L_{\alpha, \beta}$, $L_{\alpha, \beta} \delta_\alpha = \pm \delta_\alpha^{\dagger} L_{\alpha, \beta}$ (we have a plus if $dim(M)$ is odd, a minus otherwise) and that $L_{\alpha, \beta}^\dagger = L_{\alpha, \beta}$. As consequence of this fact we have that $\tau L_{\alpha, \beta}$ is a self-adjoint operator.
	\\Morover we have that $R^{\dagger}_\beta R_\beta = L_{0, \beta}$, so $L_{0, \beta}$ is $\mathcal{L}^2$-invertible if $\beta \neq 0$ and the same holds for $L_{\alpha, \beta}$ if $\vert\alpha\vert$ is small enough. In this case we have the following well-defined operator
	\begin{equation}
	S_{\alpha, \beta} := \frac{\tau L_{\alpha, \beta}}{\vert\tau L_{\alpha, \beta}\vert}.
	\end{equation}
	Observe that $S_\alpha$ is invertible (it is in particular an involution). We also have that $S_\alpha^{\dagger} = S_\alpha$.
	\\Consider, now $U_\alpha := (\vert\tau L_\alpha\vert)^{\frac{1}{2}}$. Since $\vert\tau L_\alpha\vert$ is invertible, also $U_\alpha$ is invertible.
	\begin{defn}
		If the dimension of $M$ is even, we will call \textbf{perturbed signature operator} the operator
		\begin{equation}
		D_{\alpha, \beta} := U_{\alpha, \beta}(\delta_\alpha - S_{\alpha, \beta} \delta_\alpha S_{\alpha, \beta} )U_{\alpha,\beta}^{-1}.
		\end{equation}
		if $M$ has dimension odd, we have that
		\begin{equation}
		D_{\alpha,\beta} := -i U_{\alpha,\beta}( S_{\alpha,\beta} \delta_\alpha +\delta_\alpha S_{\alpha, \beta} )U_{\alpha, \beta}^{-1}.
		\end{equation}
	\end{defn}
	\subsection{$L^2$-invertibility of $D_{\alpha,1}$}
	\begin{lem}
		Let us consider two Riemannian manifolds of bounded geometry $(M,g)$ and $(N,h)$ and let $f$ be a smooth and lipschitz  uniform homotopy equivalence. Then we have that $ker(\delta_{\alpha}) = im(\delta_{\alpha})$ in $\mathcal{L}^2(M \sqcup N)$.
	\end{lem}
	\begin{proof}
		Since $T_f d = d T_f$, we have that $T_f$ gives a map between the complexes
		\begin{equation}
		0 \xrightarrow{d_N}....\xrightarrow{d_N}dom(d_N)^{k-1} \xrightarrow{d_N} dom(d_N)^k \xrightarrow{d_N} dom(d_N)^{k+1} \xrightarrow{d_N} ...
		\end{equation}
		and
		\begin{equation}
		0 \xrightarrow{-d_M}....\xrightarrow{-d_M}dom(d_M)^{k-1} \xrightarrow{-d_M} dom(d_M)^k \xrightarrow{-d_M} dom(d_M)^{k+1} \xrightarrow{-d_M} ...
		\end{equation}
		In particular, we have that $\delta_\alpha$ is a mapping cone over these chains. Now, since $T_f$ is an isomorphism in $\mathcal{L}^2$-cohomology (Corollary 4.17 of \cite{Spes}), this means that $\delta_\alpha$ is acyclic, i.e.
		\begin{equation}
		ker(\delta_{\alpha}) = im(\delta_{\alpha}).
		\end{equation}
	\end{proof}
	\begin{prop}
		The operator $D_{\alpha, \beta}$ is injective if $\vert\alpha\vert$ is small enough and $\beta \neq 0$.
	\end{prop}
	\begin{proof}
	This is a direct computation. First we define a \textit{perturbed} scalar product on $\mathcal{L}^2(M \sqcup N )$ posing $\langle \cdot , \cdot \rangle_{\alpha \beta} := \langle \cdot , \vert \tau L_{\alpha \beta} \vert \cdot \rangle $ where $\langle \cdot , \cdot \rangle$ is the standard scalar product on $\mathcal{L}^2(M \sqcup N )$. Second we prove that $ \vert \tau L_{\alpha \beta} \vert$,  $\tau L_{\alpha \beta}$ and $S_{\alpha,\beta}$ are self-adjoint with respect $\langle \cdot ; \cdot \rangle_{\alpha \beta}$. Then we prove that $\delta_\alpha^* := \pm S_{\alpha,\beta} \delta_\alpha S_{\alpha,\beta}$ is the adjoint\footnote{we have a $-$ if the dimension of $M$ is even and a $+$ otherwise} of $\delta_\alpha$. Finally we conclude by proving that 
	\begin{equation}
	\hat{D}_{\alpha,\beta} := \delta_\alpha \pm \delta_\alpha^*
	\end{equation}
	is injective. Indeed $ker(\hat{D}_{\alpha,\beta}) = ker(\delta_\alpha) \cap ker(\delta_\alpha^*) = \{0\}$. The injectivity of $D_{\alpha,\beta}$ immediately follows because it is composition of injective oprators. Calculations can be found in Proposition 2.3.2 of \cite{thesis}.
	\end{proof}
	\begin{rem}
		The operator $D_{\alpha, \beta}$ is self-adjoint with respect to the standard scalar product.
		Indeed we have that the even perturbed signature operator is
		\begin{equation}
		D_{\alpha, \beta} =  U_{\alpha, \beta} \hat{D}_{\alpha,\beta} U^{-1}_{\alpha, \beta}.
		\end{equation}
		and the odd perturbed signature operator is
		\begin{equation}
		D_{\alpha, \beta} =  U_{\alpha, \beta} S_{\alpha, \beta} \hat{D}_{\alpha,\beta} U^{-1}_{\alpha, \beta}.
		\end{equation}
		We know that $\hat{D}_{\alpha,\beta}$ is a self-adjoint operator respect to the scalar product $\langle \cdot, \cdot \rangle_{\alpha,\beta}$. Then the self-adjointeness of $D_{\alpha, \beta}$ follows since a direct calculation and by using that $U_{\alpha, \beta} = \vert \tau L_{\alpha, \beta} \vert U_{\alpha, \beta}^{-1}$ is self-adjoint.
	\end{rem}
	\begin{prop}
		The operator $D_{\alpha,\beta}$ is $L^2$-invertible if $\alpha$ is small enough and $\beta \neq 0$.
	\end{prop}
	\begin{proof}
		Since $U_{\alpha,\beta}$ and $S_{\alpha,\beta}$ are $L^2$-invertible, it is sufficient to prove that 
		$\hat{D}_{\alpha,\beta} := \delta_\alpha + \delta_\alpha^*$ is invertible.
		\\Since $\hat{D}_{\alpha,\beta}$ is self-adjoint, its spectrum can be decomposed in \textit{essential spectrum} given by the subset of $\lambda$s such that
		\begin{equation}
		\hat{D}_{\alpha,\beta} - \lambda Id
		\end{equation}
		is not a Fredholm operator and in \textit{discrete spectrum} which is the subset given by the eigenvalues with finite multiplicity.
		\\Since $\hat{D}_{\alpha,\beta}$ is injective, the zero can't be in in discrete spectrum.
		\\Now, using the Theorem 2.4. of \cite{Bruning} we have that, since $\delta_\alpha$ is acyclic, then   zero can't be in the essential spectrum of $\hat{D}_{\alpha,\beta}$. Then $\hat{D}_{\alpha,\beta}$ is invertible and the same holds for $D_{\alpha,\beta}$.
	\end{proof}
	\subsection{The perturbation in $C^*_f(M \sqcup N)^\Gamma$}\label{pert}
	\begin{prop}\label{vv}
		The perturbation $D_{\alpha, \beta} - D$ is an operator of $C_f^*(M\sqcup N)^\Gamma$.
	\end{prop}
	\begin{proof}
		Observe that
		\begin{equation}
		\begin{split}
		D_{\alpha, \beta} - D &= U_{\alpha, \beta}(\delta_\alpha - S_{\alpha, \beta} \delta_\alpha S_{\alpha, \beta})U_{\alpha, \beta}^{-1} - d_{M \sqcup N} + \tau d_{M \sqcup N} \tau\\
		&= U_{\alpha, \beta}(\delta_\alpha - d_{M \sqcup N})U_{\alpha, \beta}^{-1} + (U_{\alpha, \beta}d_{M \sqcup N}U_{\alpha, \beta}^{-1} - d_{M \sqcup N}) \\
		&+U_{\alpha, \beta}(S_{\alpha, \beta}\delta_\alpha S_{\alpha, \beta} - \tau d_{M \sqcup N} \tau )U_{\alpha, \beta}^{-1} + (U_{\alpha, \beta}(\tau d_{M \sqcup N} \tau )U_{\alpha, \beta}^{-1} -\tau d_{M \sqcup N} \tau).
		\end{split}
		\end{equation}
		We know that $\delta_\alpha - \delta = \alpha T_f^\dagger$ is in $C_f^*(M\sqcup N)^\Gamma$.
		\\In the first step we show that $S_{\alpha, \beta} = \tau + H_{\alpha, \beta}$, $U_{\alpha, \beta} = 1 + G_{\alpha, \beta}$ and $U_{\alpha, \beta}^{-1} = 1 + K_{\alpha, \beta}$ where $H_{\alpha, \beta}$, $G_{\alpha, \beta}$, $K_{\alpha, \beta}$ are operators in $C_f^*(M \sqcup N)^\Gamma$.
		\\In order to prove this first step we follow Proposition 2.1.11. of \cite{vito}. Observe that $L_{\alpha, \beta} = 1 + Q_{\alpha,\beta}$, with $Q_{\alpha,\beta}$ in $C_f^*(M \sqcup N)^\Gamma$.
		\\This means that
		\begin{equation}
		\vert\tau L_{\alpha, \beta}\vert = \sqrt{1 + R_{\alpha, \beta}},
		\end{equation}
		where $R_{\alpha, \beta} = 2Q_{\alpha,\beta} + Q_{\alpha,\beta}^2$ in $C_f^*(M \sqcup N)^\Gamma$. Now, let us consider the complex function $g(z) := \sqrt{1 + z} - 1$. Since $g(0) = 0$, we have that
		\begin{equation}
		g(z) = az + zh(z)z
		\end{equation}
		where $h$ is a holomorphic function and $a$ is a number. Then we have that
		\begin{equation}
		\vert\tau L_{\alpha, \beta}\vert = 1 + g(R_{\alpha, \beta}) = 1 + V_{\alpha, \beta}
		\end{equation}
		with $V_{\alpha, \beta}$ in $C^*(M \sqcup N)^\Gamma$. Observe that
		\begin{equation}
		U_{\alpha, \beta} = \sqrt{\vert\tau L_{\alpha, \beta}\vert}
		\end{equation}
		Applying the same argument we obtain that
		\begin{equation}
		U_{\alpha, \beta} = 1 + G_{\alpha, \beta}.
		\end{equation}
		where $G_{\alpha, \beta}$ is in $C_f^*(M \sqcup N)^\Gamma$. Observe the operator
		\begin{equation}
		Z_{\alpha, \beta} = \vert\tau L_{\alpha, \beta}\vert^{-1} - 1
		\end{equation}
		is an operator of $D_f^*(M\sqcup N)^\Gamma$ since $\vert\tau L_{\alpha, \beta}\vert$ is in $D_f^*(M\sqcup N)^\Gamma$. In particular we have that
		\begin{equation}
		\begin{cases}
		\vert\tau L_{\alpha, \beta}\vert \circ \vert\tau L_{\alpha, \beta}\vert ^{-1} = (1 + V_{\alpha,\beta})(1 + Z_{\alpha, \beta}) =  1\\
		\vert\tau L_{\alpha, \beta}\vert ^{-1}  \circ \vert\tau L_{\alpha, \beta}\vert = (1 + Z_{\alpha, \beta})(1 + V_{\alpha,\beta}) = 1.
		\end{cases}
		\end{equation}
		and so we obtain that
		\begin{equation}
		Z_{\alpha, \beta} = - V_{\alpha, \beta} - V_{\alpha, \beta}Z_{\alpha, \beta} = - V_{\alpha, \beta} - Z_{\alpha, \beta} V_{\alpha, \beta}.
		\end{equation}
		Then we have that $Z_{\alpha, \beta}$ is in $C_f^*(M \sqcup N)^\Gamma$.
		\\Now, since
		\begin{equation}
		U_{\alpha, \beta}^{-1} = \sqrt{\vert\tau L_{\alpha, \beta}\vert^{-1}}
		\end{equation}
		then we have that
		\begin{equation}
		U_{\alpha, \beta}^{-1} = 1 + K_{\alpha, \beta}.
		\end{equation}
		where $K_{\alpha, \beta}$ is in $C_f^*(M \sqcup N)^\Gamma$.
		Moreover we also have that
		\begin{equation}
		\begin{split}
		S_{\alpha,\beta} &= \tau L_{\alpha, \beta} \vert\tau L_{\alpha, \beta}\vert^{-1} \\
		&= \tau(1 + Q_{\alpha, \beta})(1 + Z_{\alpha, \beta}) \\
		&= \tau + H_{\alpha, \beta},
		\end{split}
		\end{equation}
		where $H_{\alpha, \beta}$ in $C_f^*(M \sqcup N)^\Gamma$. This concludes the first step.
		\\
		\\In the second step, we have to prove that $G_{\alpha, \beta} d_{M \sqcup N}$, $H_{\alpha, \beta} d_{M \sqcup N}$, $K_{\alpha, \beta} d_{M \sqcup N}$, $d_{M \sqcup N}G_{\alpha, \beta}$, $d_{M \sqcup N}H_{\alpha, \beta}$, $d_{M \sqcup N}K_{\alpha, \beta}$,$G_{\alpha, \beta} d_{M \sqcup N} G_{\alpha, \beta}$ and $H_{\alpha, \beta} d_{M \sqcup N} K_{\alpha, \beta}$ are in $C_f^*(M\sqcup N)^\Gamma$.
		\\
		\\In order to prove this we have to observe that $L_{\alpha, \beta} = 1 + Q_{\alpha,\beta}$ where
		\begin{equation}
		Q_{\alpha, \beta} = \begin{bmatrix} \beta T_f^\dagger T_f && (1 - i \alpha \gamma y)\beta T_f^\dagger \\
		\beta T_f(1 + i \alpha \gamma y) && 0 \end{bmatrix}
		\end{equation}
		is an algebraic combination of $T_f$, $T_f^\dagger$, $T_fy$, $yT_f^\dagger$ and $\gamma$. Then we have that $Q_{\alpha, \beta}d_{M \sqcup N}$ and $d_{M \sqcup N}Q_{\alpha, \beta}$ is an operator in $C_f^*(M \sqcup N)^\Gamma$ (Corollary \ref{corn}). The same property, obviously, holds for $R_{\alpha, \beta}$.
		\\Observe that if $A$ is an operator in $C_f^*(M \sqcup N)^\Gamma$ such that $Ad$ and $dA$ are operators in $C^*_f(M \sqcup N)^\Gamma$ then also $g(A)d_{M \sqcup N}$ and $d_{M \sqcup N}g(A)$ are in $C_f^*(M \sqcup N)^\Gamma$, indeed
		\begin{equation}
		g(A) d_{M \sqcup N} = (aA + Ah(A)A) d_{M \sqcup N} = a(Ad_{M \sqcup N}) + Ah(A)(Ad_{M \sqcup N})
		\end{equation}
		and
		\begin{equation}
		d_{M \sqcup N}g(A) = d_{M \sqcup N}(aA + Ah(A)A) = a(d_{M \sqcup N}A) + (d_{M \sqcup N}A)h(A)A.
		\end{equation}
		So we have that the compositions of $V_{\alpha, \beta}$ and $G_{\alpha,\beta}$ with $d_{M \sqcup N}$ are operators in $C_f^*(M \sqcup N)^\Gamma$. Moreover, since
		\begin{equation}
		Z_{\alpha, \beta} = - V_{\alpha, \beta} - V_{\alpha, \beta}Z_{\alpha, \beta} = - V_{\alpha, \beta} - Z_{\alpha, \beta} V_{\alpha, \beta}
		\end{equation}
		then this property also holds for $Z_{\alpha, \beta}$, $K_{\alpha, \beta}$ and $H_{\alpha, \beta}$.
		Now, since $D_{\alpha, \beta} - D$ is an algebraic combination of operators in $C_f^*(M \sqcup N)^\Gamma$, the perturbation is in $C_f^*(M \sqcup N)^\Gamma$.
	\end{proof}
	\subsection{Involutions}
		Observe that the operator $W_{\alpha, \beta} := U_{\alpha,\beta}S_{\alpha, \beta}U_{\alpha, \beta}^{-1}$ is a well-defined, self-adjoint, involution whenever $\alpha \neq 0$ and $\beta=1$ or when $\alpha = 0$ and $\beta \in [0,1]$.
		\\We have that $\mathcal{L}^2(M \sqcup N)$ can be split, for any $\alpha$ and $\beta$, in
		\begin{equation}
		\mathcal{L}^2(M \sqcup N) = V_{+, \alpha, \beta} \oplus V_{-, \alpha, \beta}
		\end{equation}
		where $V_{\pm, \alpha, \beta}$ are the $\pm1$-eigenvalues of $W_{\alpha, \beta} := U_{\alpha,\beta}S_{\alpha, \beta}U_{\alpha, \beta}^{-1}$. Respect to this decomposition we have that
		\begin{equation}
		D_{\alpha, \beta} = \begin{bmatrix} 0 && D_{\alpha, \beta +} \\
		D_{\alpha,\beta -} && 0 \end{bmatrix}.
		\end{equation}
		Consider a real value $\alpha_0$ such that $D_{\alpha_0, 1}$ is invertible, we can define $\gamma: [0,1] \longrightarrow \mathbb{R}^2$ as $\gamma(t):= (\alpha(t), \beta(t))$ where
		\begin{equation}
		\alpha(t) = \begin{cases} 0 \text{   if    } t \in [0,\frac{1}{2}] \\
		2t\alpha_0 \text{    otherwise} \end{cases} \text{    and   } 
		\beta(t) = \begin{cases} 2t \text{   if    } t \in [0,\frac{1}{2}] \\
		1 \text{    otherwise} \end{cases}
		\end{equation}
		Then we consider the map $W_\gamma : [0,1] \longrightarrow D_f^*(M \sqcup N)^\Gamma$ defined as
		\begin{equation}
		W_{\gamma(t)} := U_{\gamma(t)}S_{\gamma(t)}U_{\gamma(t)}^{-1}.
		\end{equation}
		We want to prove that $W_\gamma$ is a continuous function and that for all $t_1$, $t_2$ we have that
		\begin{equation}
		W_{\gamma(t_1)} - W_{\gamma(t_2)} \in C_f^*(M \sqcup N)^\Gamma \label{stat}.
		\end{equation}
		The statement (\ref{stat}) can be checked by observing that for all $t$ in $[0,1]$ we have that $W_{\gamma(t)} - \tau$ is in $C_f^*(M \sqcup N)^\Gamma$.
		\\in oreder to check the continuity of $W_{\gamma(t)}$ in $t$, observe that for all $t$ in $[0,1]$ we have that
		\begin{equation}
		W_{t} := \sqrt{\vert\tau L_{\gamma(t)}\vert} \circ \frac{\tau L_{\gamma(t)}}{\vert\tau L_{\gamma(t)}\vert} \circ \sqrt{\vert\tau L_{\gamma(t)}\vert^{-1}}.
		\end{equation}
		We know that $\tau L_{\gamma(t)}$ is continuous in $t$ and the same holds for its square. We also know that
		\begin{equation}
		\vert\tau L_{\gamma(t)}\vert = 1 + f(1 - \tau L_{\gamma(t)}^2)
		\end{equation}
		where $f(z) = az + zh(z)z$ and $h$ is an holomorphic function. Since the properties of holomorphic functional calculus on bounded operator, we have that if $T_k \rightarrow T$ then $f(T_k) \rightarrow f(T)$. In particular we have that $\vert\tau L_{\gamma(t)}\vert$ is continuous in $t$. Moreover exactly with the same argument we also have that $\sqrt{\vert\tau L_{\gamma(t)}\vert}$ is continuous in $t$.
		\\The operator $\vert\tau L_{\gamma(t)}\vert^{-1}$ is continuous in $t$ because for all $t$ the operator $\vert\tau L_{\gamma(t)}\vert$ is invertible with bounded inverse and the function $z \rightarrow \frac{1}{z}$ is holomorphic in every bounded open set of $\mathbb{C}$ which doesn't contains the $0$. Finally also $\sqrt{\vert\tau L_{\gamma(t)}\vert^{-1}}$ is continuous. Then we have that $W_t$ is continuous in $t$. Moreover, since $[0,1]$ is compact, we have that $W_t$ is uniformly continuous, i.e. exists $C > 0$ such that
		\begin{equation}
		\vert\vert W_{t} - W_{t + \epsilon}\vert\vert \leq C \epsilon.
		\end{equation}
	\begin{lem}\label{Cos}
		Let $f: (M,g) \longrightarrow (N,h)$ be a smooth and lipschitz uniform homotopy equivalence between two Riemannian manifolds of bounded geometry. Consider the splitting given by
		\begin{equation}
		\mathcal{L}^2(M \sqcup N) = V_+ \oplus V_-
		\end{equation}
		where $V_\pm$ is the $\pm 1$-eigenspace of $\tau$. Then if $\alpha_0$ is such that $D_{\alpha_0, 1}$ is invertible, then there is an isometry $\mathcal{U}_{\alpha_0, \pm}: V_{\pm, \alpha_0, 1} \longrightarrow V_\pm$ (which implies $\mathcal{U}_{\alpha_0,\pm}^{*}\mathcal{U}_{\alpha_0,\pm} = I$) and $\mathcal{U}_{\alpha_0, \pm}$ is the restriction to $V_{\pm, \alpha_0, 1}$ of the operator $\frac{I \pm \tau}{2} + P_{\alpha_0}$ where $P_{\alpha_0}$ is an operator in $C_f^*(M \sqcup N)^\Gamma$.
	\end{lem}
	\begin{proof}
		We will prove the assertion just for the $+$ case. The minus case is exactly the same.
		\\Consider the operator $W_{\gamma(t)}$. Since it is uniformly continuous in $t$, we can divide $[0,1]$ in $N_0$ intervals $[t_i, t_{i+1}]$ where $t_0 = 0$ and $t_{N_0}= 1$ and 
		\begin{equation}
		\vert\vert W_{t_i} - W_{t_{i+1}}\vert\vert \leq 1.
		\end{equation}
		We know that for all $t$ the operator $W_t$ is an involution and we have an orthogonal decomposition $\mathcal{L}^2(M \sqcup N) = V_{+, \alpha(t), \beta(t)} \oplus V_{-, \alpha(t), \beta(t)}$, where $V_{\pm, \alpha(t), \beta(t)}$ is the $\pm 1$-eigenspace of $W_t$. The projection on $V_{\pm, \alpha(t), \beta(t)}$ can be written as
		\begin{equation}
		\frac{I \pm W_t}{2}.
		\end{equation}
		Let us now consider $F_i$ the restriction of $\frac{I + W_{t_{i}}}{2}$ to $V_{+, \alpha(t_{i+1}), \beta(t_{i+1})}$. 
		\\Our next step is to prove that $F_i: V_{+, \alpha(t_{i+1}), \beta(t_{i+1})}\longrightarrow V_{+, \alpha(t_i), \beta(t_i)}$ is invertible. In order to prove this fact we have to consider the operator $G_i$ given by the restriction of $\frac{I + W_{t_{i+1}}}{2}$ to $V_{+, \alpha(t_{i}), \beta(t_{i})}$.
		\\Consider now $H_i := W_{t_i} - W_{t_{i+1}}$. Then we have that for all $v$ in $V_{+, \alpha(t_{i+1}), \beta(t_{i+1})}$
		\begin{equation}
		\begin{split}
		G_i \circ F_i (v) &=  (\frac{I + W_{t_{i+1}}}{2})(\frac{I + W_{t_{i}}}{2})v \\
		&=  (\frac{I + W_{t_{i+1}}}{2})(\frac{I + W_{t_{i+1}} + H_i}{2})v \\
		&=  \frac{I + W_{t_{i+1}}}{2} v + (\frac{I + W_{t_{i+1}}}{4})H_iv \\
		&=  Iv + (\frac{I + W_{t_{i+1}}}{4})H_iv.
		\end{split}
		\end{equation}
		Now, since
		\begin{equation}
		\vert\vert(\frac{I + W_{t_{i+1}}}{4})H_i\vert\vert \leq \vert\vert \frac{I + W_{t_{i+1}}}{2}\vert\vert \cdot \vert\vert\frac{H_i}{2}\vert\vert \leq 1 \cdot \frac{1}{2}\vert\vert W_{t_{i+1}} - W_{t_{i}}\vert\vert \leq \frac{1}{2},
		\end{equation}
		we have that  $G_i \circ F_i$ is invertible. Then $F_i$ is injective. With exactly with the same argument one can prove that $F_i \circ G_i$ is invertible and so $F_i$ is also surjective.
		\\Let us now consider the isometry $U_i: V_{+, \alpha(t_{i+1}) , \beta(t_{i+1})} \longrightarrow V_{+, \alpha(t_i), \beta(t_i)}$ as
		\begin{equation}
		U_i:= \frac{F_i}{\vert F_i\vert}.
		\end{equation}
		Finally we can define the isometry as the following composition
		\begin{equation}
		\mathcal{U}_{\alpha_0, +} := U_0 \circ U_{1} \circ ... \circ U_n : V_{+, \alpha_0, 1} \longrightarrow V_{+}.
		\end{equation}
		Now we have to prove that $\mathcal{U}_{\alpha_0, +}$ is the restriction to $V_{+, \alpha_0, 1}$ of an operator $I + P_{\alpha_0}$ where $P_{\alpha_0}$ is an operator in $C_f^*(M \sqcup N)^\Gamma$. First we have to observe that $G_i = F_i^*$, indeed if $v$ is a vector of $V_{+, \alpha(t_{i+1}) , \beta(t_{i+1})}$ and $w$ is a vector of $V_{+, \alpha(t_i), \beta(t_i)}$ then observe that
		\begin{equation}
		\langle v, \frac{I + W_{t_{i+1}}}{2} w \rangle = \langle v, w \rangle = \langle \frac{I + W_{t_{i}}}{2} v, w \rangle.
		\end{equation}
		This means that $F_i^* F_i(v) = Iv + (\frac{I + W_{t_{i+1}}}{4})H_iv$. Then, since $H_i$ is in $C_f^*(M \sqcup N)^\Gamma$, we have that
		\begin{equation}
		F_i^* F_i = (I + L)_{\vert_{V_{+, \alpha(t_{i+1}) , \beta(t_{i+1})}}},
		\end{equation}
		where $L$ is an on operator in $C_f^*(M \sqcup N)^\Gamma$. So, exactly as we did in Proposition \ref{vv}, we can prove that
		\begin{equation}
		\vert F_i\vert^{-1} = (I + Q)_{\vert_{V_{+, \alpha(t_{i+1}) , \beta(t_{i+1})}}},
		\end{equation}
		where $Q$ is an operator in $C_f^*(M \sqcup N)^\Gamma$. This means
		\begin{equation}
		\begin{split}
		U_i &= \frac{F_i}{\vert F_i\vert} = (\frac{I + W_t}{2})(I + Q)\\
		&= (\frac{I + \tau}{2} + \frac{1}{2}H_{\alpha(t),\beta(t)})(I + Q)= \frac{I + \tau}{2} + C_f^*(M \sqcup N)^\Gamma
		\end{split}
		\end{equation}
		Then we have that
		\begin{equation}
		\mathcal{U}_{\alpha_0} = (\frac{I + \tau}{2})^n + C_f^*(M \sqcup N)^\Gamma = \frac{I + \tau}{2} + C_f^*(M \sqcup N)^\Gamma.
		\end{equation}
	\end{proof}
	\subsection{Uniform homotopy invariance of the Roe Index of signature operator}
Exactly as in the connected case, we have the following definition.
	\begin{defn}
		The \textbf{fundamental class of $D_{M \sqcup N}$} is $[D_{M \sqcup N}] \in K_{n+1}(\frac{D_f^{*}(M \sqcup N)^{\Gamma}}{C_f^{*}(M \sqcup N)^\Gamma})$ given by
		\begin{equation}
		[D_{M \sqcup N}] := \begin{cases} 
		
		[\frac{1}{2}(\chi (D_{M \sqcup N})+1)] \text{if $n$ is odd,}
		\\ [\chi(D_{M \sqcup N})_+] \text{if $n$ is even.}
		
		\end{cases}
		\end{equation}
	\end{defn}
\begin{rem}
Again the definition in the even case is well-given since we are considering $\chi(D_{M \sqcup N})_+$ in $B(H_{M \sqcup N})$ where $H_{M \sqcup N}$ is defined as in Example \ref{terna}.
\end{rem}
	\begin{defn}\label{doubleindex}
	The \textbf{Roe index} of $D_{M \sqcup N}$ is the class
	\begin{equation}
		{Ind_{Roe}(D_{M \sqcup N}) := \delta[D_{M \sqcup N}]}
	\end{equation}
	in $K_{n}(C_f^{*}(M \sqcup N)^{\Gamma})$, where $\delta$ is the connecting homomorphism in K-Theory.
\end{defn}
	\begin{prop}
		Let $(M,g)$ and $(N,h)$ be two manifolds of bounded geometry and let $f:(M,g) \longrightarrow (N,h)$ be a $C^2_{b}$-map which is a smooth uniform homotopy equivalence. Then if $n$ is odd we have that in $K^*(\frac{D_f^*(M \sqcup N)^\Gamma}{C_f^*(M\sqcup N)^\Gamma})$
		\begin{equation}
		[D_{M \sqcup N}] = [\frac{1}{2}(\chi(D_{M \sqcup N}) +1)] = [\frac{1}{2}(\chi(D_{\alpha,\beta}) +1)]
		\end{equation}
		and if $n$ is even
		\begin{equation}
		[D_{M \sqcup N}] = [\chi(D_{M \sqcup N})_+] = [ \mathcal{U}_{\alpha_0,-} \chi(D_{\alpha_0,1})\mathcal{U}^*_{\alpha_0,+}].
		\end{equation} 
	\end{prop}
	\begin{proof} 
		Let us start with the odd case. It's sufficient to apply Lemma 5.8 of \cite{higroe} posing $A = D_f^*(M \sqcup N)^\Gamma$, $J = C_f^*(M \sqcup N)^\Gamma$, $D$ is the signature operator, $g= \chi$ and $C = D_{\alpha, \beta} - D$. Then we obtain that
		\begin{equation}
		\chi(D_{\alpha,\beta}) - \chi(D) \in C^*_f(M \sqcup N)^\Gamma
		\end{equation}
		and so they are the same element in $\frac{D_f^*(M \sqcup N)^\Gamma}{C_f^*(M \sqcup N)^\Gamma}$ and they induce the same element in $K$-theory.
		\\
		\\For the even case it is sufficient, because of \ref{Cos}, to remind that
		\begin{equation}
		\mathcal{U}_{\alpha_0,\pm} - \frac{I \pm \tau}{2} \in C_f^*(M \sqcup N)^\Gamma
		\end{equation}
		Then, applying again Lemma 5.8 of \cite{higroe}, we also know that
		\begin{equation}
		\chi(D_{\alpha_0,1}) - \chi(D) \in C_f^*(M \sqcup N)^\Gamma
		\end{equation}
		and this means that
		\begin{equation}\label{poi}
		\chi(D_{M \sqcup N})_+ - \mathcal{U}_{\alpha_0,-} \chi(D_{\alpha_0,1})\mathcal{U}^*_{\alpha_0,+}
		\end{equation}
		is in $C_f^*(M \sqcup N)^\Gamma$. Then we conclude applying Remark \ref{girl}. Indeed the perturbation (\ref{poi}) can be seen in
		\begin{equation} 
			C^*_{\rho_{M \sqcup N}}(M \sqcup N, H_{M \sqcup N})^\Gamma
		\end{equation}
	where $H_{M \sqcup N}$ and $\rho_{M \sqcup N}$ are defined as in Example \ref{terna}.
	\end{proof}
	Observe that since $D_{\alpha_0, 1}$ is invertible, one can choose as function $\chi$ a function which is $1$ on $spec(D) \cap (0, + \infty)$ e $\chi \equiv -1$ on $spec(D) \cap (-\infty, 0)$. This means that the following definition is well-given.
	\begin{defn}
		The \textbf{$\rho_f$-class} of $D$ is the class of $K_{n+1}(D_f^*(M \sqcup N)^\Gamma)$
		\begin{equation}
		\rho_f = \begin{cases} 
		[\frac{1}{2}(\chi (D_{\alpha_0, 1})+1)] \text{if n is odd,}
		\\ [(\mathcal{U}_{\alpha_0,-}) \chi(D_{\alpha_0,1})(\mathcal{U}^*_{\alpha_0,+})] \text{if n is even.}
		\end{cases}
		\end{equation} 
	\end{defn}
	\begin{prop}
		Let $(M,g)$ and $(N,h)$ be two manifolds of bounded geometry and consider a group $\Gamma$ acting FUPD by isometries on $M$ and $N$. Let $f:(M,g) \longrightarrow (N,h)$ be a $\Gamma$-equivariant smooth and lipschitz uniform homotopy equivalence which preserves the orientation. Then if $f$ is a $C^2_{b,u}$-map, then the Roe index of $M \sqcup N$ is zero.
	\end{prop}
	\begin{proof}
		Consider the long exact sequence in $K$-theory
		\begin{equation}
		.... \xrightarrow{i_\star} K_n(D_f^*(M \sqcup N)^\Gamma) \xrightarrow{p_\star} K_n(\frac{D_f^*(M \sqcup N)^\Gamma}{C_f^*(M \sqcup N)^\Gamma}) \xrightarrow{\delta} K_{n+1}(C_f^*(M \sqcup N)^\Gamma) \longrightarrow ...
		\end{equation}
		Now, we know that $Ind_{Roe}(D_{M \sqcup N}) = \delta([D_{M \sqcup N}]$. But in this case we know that
		$\delta([D_{M \sqcup N}] = \delta(p_\star(\rho_f(D_{M \sqcup N}))) = 0$. Then the Roe index vanishes.
	\end{proof} 
	We are ready now to prove the uniform homotopy invariance of the Roe index of the signature operator.
	\begin{thm}
		Let $(M,g)$ and $(N,h)$ be two manifolds of bounded geometry. Let $\Gamma$ be a group acting uniformly proper, discontinuous and free on $M$ and $N$ by orientation-preserving isometries. Consider $f:(M,g) \longrightarrow (N,h)$ a $\Gamma$-equivariant uniform homotopy equivalence which preserves the orientations. Then
		\begin{equation}
		f_\star(Ind_{Roe}(D_M)) = Ind_{Roe}(D_N).
		\end{equation}
	\end{thm}
	\begin{proof}
		Since $\Gamma$ acts FUPD, all $\Gamma$-equivariant uniform maps are coarsely approximable by $\Gamma$-equivariant $C^2_b$-maps (Proposition 1.7 of \cite{Spes}). So we consider $f$ as a $C^2_b \cap C^{\infty}$ uniform homotopy equivalence. 
		\\Consider a $\Gamma$-equivariant isometry $V: \mathcal{L}^2(N) \longrightarrow \mathcal{L}^2(M)$ which uniformly covers $f$. We know that such an operator exists by Proposition 4.3.5. of \cite{siegel}. Then we know that
		\begin{equation}
		f_\star([D_M]) = Ad_{V,\star}[D_M] = [V D_M V^*],
		\end{equation}
		where we have the $u$ if $n$ is odd and we haven't it if $n$ is even.
		\\Let us consider now $C_f^*(M \sqcup N)^\Gamma$. Consider the decomposition
		\begin{equation}
		D_f^*(M \sqcup N)^\Gamma = \bigoplus\limits_{X,Y = M,N} D^*(M \sqcup N)_{X,Y}^\Gamma,
		\end{equation}
		where
		\begin{equation}
		 D^*(M \sqcup N)_{X,Y}^\Gamma := D_f^*(M \sqcup N)^\Gamma \cap B(\mathcal{L}^2(X), \mathcal{L}^2(Y)).
		\end{equation}
		Let us define now the operator $H: D_f^*(M \sqcup N)^\Gamma \longrightarrow D^*(N)^\Gamma$ defined  as
		\begin{equation}
		H(\begin{bmatrix} A_{MM} && A_{MN} \\ A_{NM} && A_{NN} \end{bmatrix})=  VA_{MM}V^* + A_{NN}.
		\end{equation}
		where $A_{XY} \in  D^*(M \sqcup N)_{X,Y}^\Gamma$. $H$ is a $*$-homomorphism, and so it means that $H$ induce a map between the $K$-theory groups. In particular we have that $H(C_f^*(M \sqcup N)^\Gamma) \subset C^*(N)^\Gamma$. This means that $H$ also induces some morphism between the Coarse algebra and the quotient algebras. Let us consider $H : \frac{D_f^*(M \sqcup N)^\Gamma}{C_f^*(M \sqcup N)^\Gamma} \longrightarrow \frac{D^*(N)^\Gamma}{C^*(N)^\Gamma}$.
		\\Observe that, if $n$ is odd then in  we have that
		\begin{equation}
		\begin{split}
		\frac{1}{2} (\chi(D_{M \sqcup N}) + 1) &= \begin{bmatrix} \frac{1}{2} (\chi(D_M) + 1) && 0 \\ 0 && - \frac{1}{2} (\chi(D_N) + 1)  \end{bmatrix}
		\\
 &= \begin{bmatrix} \frac{1}{2} (\chi(D_M) + 1) && 0 \\ 0 && 0  \end{bmatrix} - \begin{bmatrix} 0 && 0 \\ 0 && \frac{1}{2} (\chi(D_N) + 1)  \end{bmatrix} =: \tilde{D}_M - \tilde{D}_N.
		\end{split}
		\end{equation}
		and so in $K_{n+1}(\frac{D_f^*(M \sqcup N)^\Gamma}{C_f^*(M \sqcup N)^\Gamma})$ we have
		\begin{equation}
		[D_{M \sqcup N}] = [\tilde{D}_M] - [\tilde{D}_N].
		\end{equation}
		Let us suppose that $n$ is even, we have that
		\begin{equation}
		\begin{split}
		\chi(D_M)_+ &= \begin{bmatrix} \chi(D_M)_+ && 0 \\ 0 && \chi(D_N)_+  \end{bmatrix}
		= \begin{bmatrix} \chi(D_M)_+ && 0 \\ 0 && 1  \end{bmatrix} \cdot \begin{bmatrix} 0 && 0 \\ 0 &&  \chi(D_N)_+ \end{bmatrix} \\
 &=: \tilde{D_M}\cdot \tilde{D_N}.
		\end{split}
		\end{equation}
		and so, again, in $K_{n+1}(\frac{D_f^*(M \sqcup N)^\Gamma}{C_f^*(M \sqcup N)^\Gamma})$ we have\footnote{We are considering $K_1$ as generated by unitaries and so, if $u,v$ are two unitaries we have that $[uv] = [u] + [v]$. See pg 106 of \cite{Anal}.}
		\begin{equation}
		[D_{M \sqcup N}] = [\tilde{D}_M] - [\tilde{D}_N].
		\end{equation}
		Then, in both cases, we obtain that
		\begin{equation}
		H_\star[D_{M \sqcup N}] = H_{\star}[\tilde{D}_M] - H_{\star}[\tilde{D}_N] = f_\star[D_M] - [D_N].
		\end{equation}
		Then we conclude by applying the connecting homomorphism, indeed
		\begin{equation}
		    \begin{split}
		      0 &= H_\star \delta_{M \sqcup N}[D_{M \sqcup N}] \\
		      &= H_\star \delta_{M \sqcup N} [\tilde{D}_M] - H_\star \delta_{M \sqcup N}[\tilde{D}_N] \\
		      &= \delta_N  H_\star [\tilde{D}_M] - \delta_N  H_\star[\tilde{D}_N] \\
		      &= \delta_N f_\star[D_M] - \delta_N [D_N]\\
		      &= f_\star \delta_M[D_M] - \delta_N [D_N]\\
		      &= f_\star (Ind_{Roe}(D_M)) - Ind_{Roe}(D_N).
		    \end{split}
		\end{equation}
		\end{proof}

	\phantomsection
	\bibliographystyle{sapthesis} 
	\end{document}